\documentclass[10pt]{amsart}

\usepackage[all]{xy}
\usepackage[latin1]{inputenc} 
\usepackage[T1]{fontenc}
\usepackage[dvips]{graphics,graphicx}
\usepackage{amsfonts,mathabx}
\usepackage{amssymb}
\usepackage{amsmath}
\usepackage{mathrsfs, mathtools}
\usepackage{esint}
\usepackage{array, makecell} 
\allowdisplaybreaks

\usepackage{amsbsy,
	amsopn,
	amscd, 
	amsxtra, 
	amsthm,
	verbatim}
\usepackage{upref}
\usepackage{color}
\usepackage{enumitem}
\usepackage[colorlinks,
linkcolor=red,
anchorcolor=red,
citecolor=red
]{hyperref}

\usepackage{thmtools}
\usepackage{caption}
\usepackage{subcaption}
\usepackage{verbatim}
\usepackage{booktabs}
\usepackage{longtable}

\newtheorem{theorem}{Theorem}
\newtheorem{lemma}{Lemma}[section]
\newtheorem{proposition}[lemma]{Proposition}
\newtheorem*{proposition*}{Proposition}
\newtheorem{corollary}[lemma]{Corollary}
\newtheorem*{corollary*}{Corollary}

\newtheorem{assumption}{Assumption}

\newtheorem{remark}[lemma]{Remark}

\newcommand{\thmref}[1]{Theorem~\ref{#1}}

\newcommand{\propref}[1]{Proposition~\ref{#1}}

\newcommand{\secref}[1]{Sec.~\ref{#1}}

\usepackage{geometry}
\usepackage{float}
\newcommand{\vertiii}[1]{{\left\vert\kern-0.25ex\left\vert\kern-0.25ex\left\vert #1 \right\vert\kern-0.25ex\right\vert\kern-0.25ex\right\vert}}

\usepackage{todonotes}

\newcommand{\ie}{{i.e.}}
\newcommand{\eg}{{e.g.}}
\makeatletter
\newcommand*{\rom}[1]{\expandafter\@slowromancap\romannumeral #1@}
\makeatother

\newcommand{\wb}[1]{\widebar{#1}}

\providecommand{\keywords}[1]{\textbf{\textit{Keywords: }} #1}

\newcommand{\ud}{\,\mathrm{d}}
\newcommand{\rd}{\mathrm{d}}

\newcommand{\RR}{\mathbb{R}}
\newcommand{\Real}{\mathbb{R}}

\newcommand{\vectbf}[1]{\boldsymbol{#1}}
\newcommand{\Lap}{\mathscr{L}}

\newcommand{\eps}{\epsilon}
\newcommand{\abs}[1]{\lvert#1\rvert}
\newcommand{\Abs}[1]{\left\lvert#1\right\rvert}
\newcommand{\norm}[1]{\lVert#1\rVert}
\newcommand{\Norm}[1]{\left\lVert#1\right\rVert}

\newcommand{\Inner}[2]{\left(#1, #2\right)}

\renewcommand{\Re}{\mathfrak{Re}}

\newlength{\leftstackrelawd}
\newlength{\leftstackrelbwd}
\def\leftstackrel#1#2{\settowidth{\leftstackrelawd}%
{${{}^{#1}}$}\settowidth{\leftstackrelbwd}{$#2$}%
\addtolength{\leftstackrelawd}{-\leftstackrelbwd}%
\leavevmode\ifthenelse{\lengthtest{\leftstackrelawd>0pt}}%
{\kern-.5\leftstackrelawd}{}\mathrel{\mathop{#2}\limits^{#1}}}

\def\bigl{\mathopen\big}

\def\bigr{\mathclose\big}

\newcommand{\id}{\mathcal{I}}

\newcommand{\opA}{\mathcal{A}}

\newcommand{\opL}{\mathcal{L}}

\newcommand{\be}{Bakry-{\'E}mery}

\newcommand{\opLham}{\mathcal{L}_{\text{ham}}}
\newcommand{\opLfd}{\mathcal{L}_{\text{FD}}}
\newcommand{\stationary}{\rho_{\infty}}
\newcommand{\poin}{Poincar{\'e}}
\newcommand{\Rconst}{\mathsf{R}_{\text{ham}}}
\newcommand{\lyap}{\mathsf{E}}

\begin{document}
\title[On explicit $L^2$-convergence rate estimate]{On explicit $L^2$-convergence rate estimate for underdamped Langevin dynamics}
\author{Yu Cao}
\address{Institute of Natural Sciences \& School of Mathematical Sciences, Shanghai Jiao Tong University, Shanghai 200240, China}
\email{yucao@sjtu.edu.cn}
\author{Jianfeng Lu}
\address{Department of Mathematics, Department of Physics, and Department of Chemistry, Duke University, Durham NC 27708, USA}
\email{jianfeng@math.duke.edu}
\author{Lihan Wang}
\address{Department of Mathematical Sciences, Carnegie Mellon University, Pittsburgh PA 15213, USA}
\email{lihanw@andrew.cmu.edu}

\date{\today}

	
	\begin{abstract}
      We provide a refined explicit estimate of exponential decay rate of
      underdamped Langevin dynamics in $L^2$ distance, based on a framework developed in \cite{armstrong2019variational}.  To achieve
      this, we first prove a Poincar\'{e}-type inequality with Gibbs
      measure in space and Gaussian measure in momentum. Our
      estimate provides a more explicit and simpler expression of
      decay rate; moreover, when the potential is convex with
      Poincar\'{e} constant $m \ll 1$, our estimate shows the 
      decay rate of $O(\sqrt{m})$ after optimizing the
      choice of friction coefficient, which is much faster 
      than $m$ for the overdamped Langevin
      dynamics.

    \end{abstract}


	\maketitle


	\section{Introduction}

    We consider the convergence rate for the following
    \emph{underdamped Langevin dynamics}
    $({x}_t, {v}_t)\in \Real^d\times \Real^d$, given by
	\begin{equation}
		\label{eqn::langevin}
		\left\{
		\begin{aligned}
			\rd x_t &=  {v}_t\ud t\\
			\rd {v}_t &=  -\nabla U({x}_t)\ud t - \gamma {v}_t\ud t + \sqrt{2\gamma}\ud {W}_t,\\
		\end{aligned}\right.
	\end{equation}
	where $U(x)$ is the potential energy, $\gamma>0$ is the friction
    coefficient, and ${W}_t$ is a $d$-dimensional standard Brownian
    motion; the mass and temperature are set to be $1$ for simplicity. The law of the process 
    \eqref{eqn::langevin}, $\rho(t,x,v)$, satisfies the kinetic Fokker-Planck equation
	\begin{align}
		\label{eqn::fp}
		\partial_t \rho = - v \cdot \nabla_x \rho + \nabla_x U \cdot \nabla_v\rho + \gamma\bigl(\Delta_v \rho + \nabla_v \cdot (v \rho)\bigr).
	\end{align}
	It is well-known (see for example \cite[Proposition 6.1]{pavliotis2014stochastic}) that under mild assumptions, \eqref{eqn::fp} admits a unique stationary density function given by
	\begin{align}
		\ud\stationary(x, v) = \ud \mu(x) \ud \kappa(v),
	\end{align}
	where
	\begin{align*}
		\ud \mu(x) = \frac{1}{Z_U}e^{-U(x)}\ud x, \qquad \ud \kappa(v) = \frac{1}{(2\pi)^{d/2}} e^{-\frac{\abs{v}^2}{2}} \ud v, \qquad Z_U = \int_{\Real^d} e^{- U(x)}\ud x.
	\end{align*}
When $\gamma\rightarrow\infty$, the rescaled dynamics  $x^{(\gamma)}_t := x_{\gamma t}$ converges to the \emph{Smoluchowski SDE}, also known as the \emph{overdamped Langevin dynamics} (see \eg{}, \cite[Sec.~ 6.5]{pavliotis2014stochastic}), which is given by
	\begin{align*}
	\rd x^{(\infty)}_t = -\nabla U(x^{(\infty)}_t)\ud t + \sqrt{2}\ud B_t.
	\end{align*}
An equivalent formalism of \eqref{eqn::fp} is the following \emph{backward Kolmogorov equation},
\begin{align}
		\label{eqn::opL}
		\begin{split}
		\partial_t f &= \opL f,\qquad  \opL = \opLham + \gamma \opLfd, \qquad  f(0, x, v) = f_0(x,v),
		\end{split}
	\end{align}
	where $\opLham$ is the Hamiltonian transport operator and $\opLfd$ is the fluctuation-dissipation term
	\begin{equation}
		\label{eqn::Lham_Lfd}
		\left\{
		\begin{aligned}
			\opLham &= v \cdot \nabla_x - \nabla_x U \cdot \nabla_v \\
			\opLfd &= \Delta_v - v \cdot \nabla_v.
		\end{aligned}\right.
	\end{equation}
    Indeed, \eqref{eqn::opL} could be derived from \eqref{eqn::fp} by considering $\rho(t, x, v) = f(t, x, -v) \stationary(x, v)$ \cite{pavliotis2014stochastic};
since by $L^2$-duality, $\Norm{\rho - \stationary}_{L^2(\stationary^{-1})} \equiv \Norm{f - \int f\ud\stationary}_{L^2(\stationary)}$,  the exponential convergence of the solution $\rho(t,\cdot,\cdot)$ of \eqref{eqn::fp} to $\rho_\infty$ is equivalent to the exponential decay of $f(t, \cdot, \cdot)$ to zero, provided that $\int f_0\ud \stationary = 0$. Similarly, one could obtain the backward Kolmogorov equation for the overdamped Langevin dynamics, which is given by 
\begin{align}
\label{eqn::generator_overdamped}
    \partial_t  h =  -\nabla_x U \cdot \nabla_x h + \Delta_x h, \qquad h(0, x) = h_0(x).
\end{align}
If $\mu$ satisfies a \poin{} inequality, one could show that the generator in the above equation \eqref{eqn::generator_overdamped} is self-adjoint and coercive with respect to $L^2(\mu)$. As a consequence, if $\int h_0\ud \mu = 0$, then $h(t, x)$ decays to zero exponentially fast as $t\rightarrow\infty$, see for example \cite[Theorem 4.2.5]{bakry_analysis_2014}.

\smallskip

Unlike the generator of \eqref{eqn::generator_overdamped}, the
generator $\opL$ in \eqref{eqn::opL} for the underdamped Langevin is
not uniformly elliptic. As a result, proving the exponential
convergence of $\rho(t, \cdot, \cdot)$ to the equilibrium
$\stationary$ is more challenging. With extensive works throughout the years, the
exponential convergence of the underdamped Langevin dynamics is now
better understood in various norms (see \secref{subsec::review} below
for a review).

\smallskip

Our goal in this work is to provide an explicit estimate of the
decay rate in $L^2$ for the semigroup in \eqref{eqn::opL}, based on a
framework proposed in
\cite{armstrong2019variational} which implicitly uses H\"{o}rmander's bracket conditions \cite{hormander1967hypoelliptic}. In particular, under some mild
assumptions of $U$, we obtain explicit estimates for some universal constant $C>1$ independent of $U,\gamma,d$ and some $\nu> 0$ such that for any possible $f=f(t,x,v)$ satisfying \eqref{eqn::opL}
and $\int_{} f_0 \ud\stationary = 0$, we have
	\begin{align}
		\label{eqn::exp_decay}
		\Norm{f(t, \cdot, \cdot)}_{L^2(\stationary)} \le Ce^{-\nu t}\Norm{f_0}_{L^2(\stationary)} .
	\end{align}
	
	\smallskip
	
    In the rest of this section, we will first present in
    \secref{subsec::main_result} our assumptions and main results. Next, we will briefly review
    existing approaches to study the exponential convergence of
    \eqref{eqn::opL} (or equivalently \eqref{eqn::fp}) in
    \secref{subsec::review}. and
    compare our estimate of the decay rate $\nu$ with some
    previous works aiming at explicit estimates
    \cite{roussel2018spectral,
      dalalyan2018sampling,baudoin2021gamma, ma2021there}. We would like to comment here that convergence results are also obtained in earlier works \cite{dolbeault2013exponential, grothaus2016hilbert}, although their rates are only explicit in $\gamma$.

	\subsection*{Notations}
Throughout the paper 
we assume $I$ to be the
time interval $(0,T)$, and we use $\ud \lambda(t)=\frac{1}{T}\chi_{(0,T)}(t)\ud t$  to denote the rescaled Lebesgue measure on $I$ so that $\ud \lambda(t)$ denotes a probability measure. For any probability measure $\rho$, we use $L^2(\rho)$ (and similarly $H^1(\rho),H^2(\rho)$) to denote the standard Sobolev spaces, and $H^{-1}(\rho)$ to denote the dual space of $H^1(\rho)$. For the Gaussian probability measure $\kappa$ in velocity space, we also use $L^2_\kappa$, $H^1_\kappa, \, H^{-1}_\kappa$ to denote the corresponding spaces.  Moreover, we use $H_0^1(\lambda\otimes\mu)$ to denote the $H^1(\lambda \otimes\mu)$ functions that vanish at both time boundaries $t=0$ and $t=T$. 
By abuse of notation, we denote the canonical pairing $\langle \cdot, \cdot\rangle_{H^{-1}(\rho),H^{1}(\rho)}$ between $f\in H^1(\rho)$ and $g\in H^{-1}(\rho)$ 
by \begin{equation*}
		\int fg\ud \rho:=\langle g, f \rangle_{H^{-1}(\rho),H^{1}(\rho)}.
	\end{equation*}
For $f\in H^{-1}(\rho)$, we use the notation $(f)_{\rho} := \langle f,1\rangle_{H^{-1}(\rho),H^1(\rho)}$. For an arbitrary Banach space $V$ and time interval $I$ equipped with Lebesgue measure $\ud \lambda(t)$, we denote by $L^p(\lambda \otimes \mu;V)$ the Banach space of functions $f(t,x,v)$ with norm  \begin{equation*}
		\|f\|_{L^p(\lambda \otimes \mu;V)}:=\Big(\int_{I\times\mathbb{R}^d} \|f(t,x,\cdot)\|^p_V \ud \lambda(t)\ud\mu(x)  \Big)^\frac{1}{p}.
	\end{equation*}
Inspired by \cite{armstrong2019variational}, we define the Banach space $$H_{hyp}^1(\lambda \otimes \mu):=\big\{
	f\in L^2(\lambda \otimes \mu;H^1_\kappa)~:~\partial_t f-\opLham f\in L^2(\lambda \otimes \mu;H_\kappa^{-1})
	\big\}.$$ 
We define a projection operator for $\phi(t,x,v)\in L^2(\lambda \otimes \stationary)$ by 
	\begin{align}
		\label{eqn:Pi_v}
		(\Pi_v \phi)(t,x) := \int_{\Real^d} \phi(t,x,v)\ud \kappa(v). 
	\end{align}
	Equivalently, $\Pi_v$ is used to obtain the marginal component of $\phi$ in $L^2(\lambda \otimes\mu)$. By slight abuse of notation, for $\phi(x, v) \in L^2(\stationary)$, we also use the same notation $\Pi_v$ to represent the similar projection, \ie, $(\Pi_v \phi)(x) := \int_{\Real^d} \phi(x, v) \ud \kappa(v)$.
The adjoints of $\nabla_x$ and $\nabla_v$ in the Hilbert space $L^2(\stationary)$ are respectively given by $\nabla_x^* F= - \nabla_x \cdot F +  \nabla_x U \cdot F$ and $\nabla_v^* F = -\nabla_v \cdot F + v \cdot F$ for any vector field $F(x,v): \Real^{2d} \rightarrow \Real^d$. Thus we can rewrite operators $\opLham$ and $\opLfd$ as 
	\begin{align}
	\label{eqn::opLham_opLfd_v2}
		\opLham = \nabla_v^* \nabla_x - \nabla_x^* \nabla_v, \qquad \opLfd = -\nabla_v^* \nabla_v.
	\end{align}
	 For time-augmented state space $I\times \RR^d$ equipped with measure $\lambda\otimes \mu$, we use the convention $\partial_{x_0}:=\partial_t$, the short-hand notation $\wb{\nabla}:=(\partial_t,\nabla_x)^\top$, and the notation $\Lap:=-\partial_{tt}+\nabla_x^*\nabla_x$ to denote the ``Laplace'' operator on $L^2(\lambda\otimes\mu)$. We use $C$ to denote a universal constant independent of all parameters that may change from line to line.

	\subsection{Assumptions and main results}
	\label{subsec::main_result}

	\begin{assumption}[Poincar\'{e} inequality for $\mu$]\label{assump:poincare}
		Assume that the potential $U(x)$ satisfies a Poincar\'{e} inequality in space 
		\begin{equation}
			\label{eqn:spatialpoincare}
			\int_{\mathbb{R}^d} \left(f-\int_{\mathbb{R}^d} f \ud\mu\right)^2\rd\mu \le \dfrac{1}{m}\int_{\mathbb{R}^d} |\nabla_x f|^2 \ud\mu, \qquad \forall f\in H^1(\mu).
		\end{equation}
	\end{assumption}
	\begin{assumption}\label{assump:hessian}
		The potential $U\in C^2(\RR^d)$, and there exist constants $M>0$ and $\delta\in(0,1)$ such that	\begin{equation}\label{eqn:stoltzcond9}
			|\nabla_x^2 U(x)|^2= \sum_{i,j=1}^d |\partial_{x_ix_j} U(x)|^2\le  M^2(d + |\nabla_x U(x)|^2), \mbox{ and } \Delta_x U(x) \le Md + \frac{\delta}{2}|\nabla_x U(x)|^2 \qquad \forall  \  x\in \mathbb{R}^d.
		\end{equation}
		for some constant $M\ge 1$.
	\end{assumption}
	\begin{assumption}\label{assump:spectral}
	    The embedding $H^1(\mu)\xhookrightarrow{} L^2(\mu)$ is compact.
	\end{assumption}
	
	\begin{remark}
	\begin{enumerate}[label=(\roman*),wide]
	\item Assumption \ref{assump:poincare} guarantees that the elliptic equation $\nabla_x^*\nabla_x u = h$ has a unique solution $u\in H^2(\mu)$ for any $h\in L^2(\mu)$ satisfying $(h)_\mu=0$ (see for example \cite[Proposition 5]{dolbeault_hypocoercivity_2015}). Hence, together with  Assumption~\ref{assump:spectral}, we derive from Fredholm alternative that $L^2(\mu)$ has an orthonormal basis $\{1\}\cup \{w_\alpha\}_{\alpha>0}$ where $w_\alpha \in H^2(\mu)$ are eigenfunctions of $\nabla_x^*\nabla_x$ with eigenvalue $\alpha^2$ for a discrete set of $\alpha>0$ (see \cite[Chapter 6]{evans2010partial} for an argument with bounded domains): \begin{equation*}
	     \nabla^*_x\nabla_x w_\alpha=\alpha^2 w_\alpha.
	 \end{equation*}
	Further, by Assumption~\ref{assump:poincare}, any eigenvalue $\alpha^2$ of $\nabla^*_x \nabla_x$ satisfies $\alpha \geq \sqrt{m}$, in fact, the smallest $\alpha$ is precisely $\sqrt{m}$, the square root of the Poincar\'e constant; the spectrum of $\nabla_x^*\nabla_x$ is unbounded from above. 
	
	\item Assumption~\ref{assump:spectral} is satisfied when \begin{equation*}
	    \lim_{|x|\to\infty} \dfrac{U(x)}{|x|^\beta}=\infty
	\end{equation*} for some $\beta>1$ (see \cite{hooton1981compact} for a proof). We would like to comment here that we require Assumption \ref{assump:spectral} only for technical purposes, more precisely in the proof of Lemma \ref{lem:truetestfn} where we used the spectral decomposition of the elliptic operator $\nabla_x^*\nabla_x$ to construct the test functions we desire. We believe that the assumption is not necessary for our main results to hold. We leave this for future research.
	
		\item Similar versions of Assumption~\ref{assump:hessian} is commonly used in the literature, see \eg{}, the books \cite{pavliotis2014stochastic,villani_hypocoercivity_2009} and the papers \cite{dolbeault_hypocoercivity_2009, dolbeault_hypocoercivity_2015}, and is satisfied when $U$ grows at most exponentially fast as $x\to\infty$. Here we adopt the more natural dimension scaling in \cite[Assumption 1]{bernard2022hypocoercivity} (in particular, we take $c_1=c_3=M$ in their setting), 
  since in the case of separable potential $U(x) = \sum_{i=1}^d u(x_i)$, this amounts to the more natural one-dimensional estimate $|u''|^2 \le M(1+|u'|^2)$. 
	
	\end{enumerate} \end{remark}
	
	\smallskip

	\begin{theorem}
		\label{thm::decayrate}
		Under Assumptions~\ref{assump:poincare}, ~\ref{assump:hessian}, and~\ref{assump:spectral}, there exist a constant $\nu > 0$ and universal constants $C,c$ independent of all parameters such that, for every $f(t,x,v)$ satisfying the
        backward Kolmogorov equation \eqref{eqn::opL} with initial condition $f_0 \in L^2(\mu;H^1_\kappa)$ and \begin{equation}\label{eqn:meanzero}(f_0)_{\stationary}=0,\end{equation} we have, for
        every $t\in (0,\infty)$,
        \begin{equation*}
          \|f(t,\cdot)\|_{L^2(\stationary)} \le C\exp(-\nu t)\|f_0\|_{L^2(\stationary)}.
		\end{equation*}
		Moreover, $\nu$ can be made explicit as 
		\begin{equation}\label{eqn:exprate}
		    \nu = \dfrac{m\gamma}{c(\sqrt{m}+R+\gamma)^2}
		\end{equation} 
		with some constant $R>0$ given by 
		\begin{enumerate}
			\item[(\romannumeral1)] If $U$ is convex, then 
			\[
			   R=0.
			\]
			\item[(\romannumeral2)] If the Hessian of $U$ is bounded from below  \begin{equation}\label{eqn:lowerbdriccicurv}
				\nabla_x^2 U(x) \ge -K\, \mathrm{Id}, \qquad \forall\, x \in \RR^d
			\end{equation} 
			for some constant $K \ge 0$, then  
			\[
			R=\sqrt{K}.
			\]
			Note that if $K = 0$, we recover the estimate in case (i). 
			\item[(\romannumeral3)] In the most general case without further assumptions,  
			\[
			R=M+M^\frac{3}{4}d^\frac{1}{4}.
			\]
			
		\end{enumerate}
	\end{theorem}

	\begin{remark}\ 
	\label{rmk::optimal_gamma}
	\begin{enumerate}[label=(\roman*),wide]
	\item If we fix $m=O(1)$, then, when $\gamma\to 0$
      (resp.~$\gamma\to \infty$), our estimate provides an estimate on
      decay rate of $O(\gamma)$
      (resp.~$O(\gamma^{-1})$). This is consistent with
      \cite{dolbeault2013exponential, grothaus2016hilbert, roussel2018spectral} and also the isotropic Gaussian case
      when $U(x)=\frac{m}{2}|x|^2$ (see Appendix \ref{app::isoquad}).
      
      \smallskip

    \item In the convex case, if we optimize with respect to $\gamma$ by choosing $\gamma=\sqrt{m}$, then 
\begin{align*}
\nu=\frac{\sqrt{m}}{4c}.
\end{align*}
As is shown in Appendix \ref{app::isoquad}, the scaling on $m$ is
optimal in the regime $m\to 0$, as it is the rate even for isotropic
quadratic potential. We refer the readers to Appendix~\ref{sec::DMS} for the corresponding results from the DMS method, with a slightly more explicit estimate compared to \cite{roussel2018spectral}. 

\smallskip

    \item In the case where condition \eqref{eqn:lowerbdriccicurv} is satisfied, \eg{} for the double well potential $U(x)=(|x|^2-1)^2$ with $K=4$, our scaling on $K$ is consistent with \cite[Theorem 1]{ledoux1994simple} and \cite[Sec. 5]{ledoux2004spectral}. Similar assumption is also used in \cite[Theorem 1]{otto2000generalization} for functional inequalities. 
	
\smallskip 

    \item It is well-known that for overdamped Langevin
  dynamics, the decay rate is simply $m$ in $L^2(\mu)$ for
  \eqref{eqn::generator_overdamped}.  By part (\romannumeral2) of this
  remark, when $m \ll 1$, the underdamped Langevin dynamics
  \eqref{eqn::langevin} could converge to its equilibrium
  $\stationary$ at a rate $O(\sqrt{m})$ for convex potentials, which is much faster than the overdamped Langevin dynamics.

	\item Due to the following relation (see \eg{}, \cite{sason_f_2016})
	\begin{align*}
      \frac{1}{\sqrt{2}}\Norm{\rho - \stationary}_{\text{TV}}  \le  \sqrt{\mathrm{KL}\bigl(\rho \,\|\,\stationary\bigr)}     \le \sqrt{\chi^2(\rho, \stationary)}  \equiv \Norm{\rho - \stationary}_{L^2(\stationary^{-1})} \equiv \Norm{f - \int f\ud\stationary}_{L^2(\stationary)}, 
	\end{align*}
 where $f = \ud \rho / \ud \stationary$, and the Talagrand inequality \cite{otto2000generalization} $W_2(\rho,\stationary) \le \sqrt{\frac{2}{C_{LSI}}\mathrm{KL}(\rho\|\stationary)}$
    where $C_{LSI}$ is the logarithmic Sobolev constant, \thmref{thm::decayrate}
    implies that $\rho(t, \cdot, \cdot)$ converges to $\stationary$
    with rate $2\nu$ in both $\chi^2$-divergence and relative entropy, and with rate $\nu$ in total variation and (if $\mu$ satisfies log-Sobolev inequality) 2-Wasserstein distance. On the other hand, our result does not imply \[d(\rho_t,\stationary) \le C\exp(-\nu t) d(\rho_0,\stationary) \] where $d(\rho,\pi) = TV(\rho,\pi), \, W_2(\rho,\pi)$ or $\mathrm{KL}(\rho\|\pi)$. It is interesting to study if one could establish the same convergence rate with Wasserstein distance (which is the same as asking if one could establish a coupling argument for our result) or relative entropy. \end{enumerate}
\end{remark}
	
Our decay estimate is based on the following Poincar\'{e}-type inequality in time-augmented space:
	\begin{theorem}\label{thm:poincare}
		Under Assumptions~\ref{assump:poincare}, ~\ref{assump:hessian}, and~\ref{assump:spectral}, there exist a universal constant $C$ independent of all parameters, and a constant $R<\infty$ (the same constant as in Theorem~\ref{thm::decayrate}) such that for every $f\in H_{hyp}^1(\lambda \otimes \mu)$, we have  \begin{multline}\label{eqn:hypopoincare}
			\|f-(f)_{\lambda \otimes \stationary}\|_{L^2(\lambda \otimes \stationary)} \le C\Bigl(\bigl(1+RT+ \frac{1}{(1-e^{-\sqrt{m}T})^2}+ \frac{R}{\sqrt{m}(1-e^{-\sqrt{m}T})^2}\bigr)\|(\id-\Pi_v) f\|_{L^2(\lambda \otimes \stationary)}   \\ +\big(\frac{1}{\sqrt{m}(1-e^{-\sqrt{m} T})}+T\big)\| \partial_t f-\opLham f\|_{L^2(\lambda \otimes \mu;H^{-1}_\kappa)} \Bigr).
		\end{multline}
	\end{theorem}
Let us give a brief introduction on the strategy of the proof, which is strongly motivated by the work of Armstrong and Mourrat \cite{armstrong2019variational}. A naive energy estimate and Gaussian Poincar\'e inequality yields \begin{align*}
    \dfrac{\ud}{\ud t} \|f(t, \cdot)\|_{L^2(\stationary)}^2  =-2\gamma
      \|\nabla_v f(t, \cdot)\|_{L^2(\stationary)}^2\le-2\gamma
      \|(\id-\Pi_v) f(t, \cdot)\|_{L^2(\stationary)}^2.
\end{align*} 
While the above establishes the $L^2$ energy decay, it does not directly yield exponential decay rate.
In particular, the energy dissipation is only present in velocity variable. However, instead of looking at single time slice, we should look at time intervals, since after time propagation, the dissipation in $v$ together with the transport terms in $x$ will lead to dissipation in $x$. Moreover, in the analysis, we are essentially treating the time variable $t$ as another space variable alongside $x$. With the help a Poincar\'e-type inequality in the time-augmented state space established in Theorem \ref{thm:poincare}, we can prove exponential convergence still using the standard energy estimate, in line with the moral ``hypocoercivity is simply coercivity with respect to the correct norm'', quoted from \cite[Page 4]{armstrong2019variational}.

\smallskip

To prove Theorem \ref{thm:poincare}, as an educated reader might realize from \cite{dolbeault_hypocoercivity_2015}, the elliptic regularity in $x$ variable plays an important role in the estimates, which in Lemma \ref{lem:ellreg} we made a mild generalization to the time-augmented space $L^2(\lambda \otimes\mu)$. However, in the proof of Theorem \ref{thm:poincare} when applying integration by parts, we need test functions that vanish at both boundary layers $t=0$ and $t=T$, which is not necessarily satisfied by the derivatives of the solution to the elliptic equation \eqref{eqn:mixedelliptic}. This is why we resort to Lemma \ref{lem:truetestfn} (also an extension of Bogovskii's operator \cite{bogovskii1979solution} to $(I\times \RR^d,\lambda\otimes\mu)$) for the solution of the divergence equation \eqref{eqn:divergence}, which is a cornerstone of this proof. In particular,  even for convex $U$, the constants in \eqref{eqn:hypopoincare} blow up as $T\to 0$, which can be traced down to the estimate of $\psi_{2,\alpha}'$ in  \eqref{eqn:psi2prest}, and thus prevents us from working on single time slices.  
	
		\subsection{A literature review and comparison}
	\label{subsec::review}
Kinetic Fokker-Planck equation was first studied by Kolmogorov \cite{kolmogoroff1934zufallige}, and was the main motivation for H\"{o}rmander's theory on hypoelliptic equations \cite{hormander1967hypoelliptic}, which gave an almost complete classification of second-order hypoelliptic operators. The earliest result regarding its exponential convergence were established in \cite{tropper1977ergodic} for potentials with bounded Hessian, which was later generalized in \cite{talay2002stochastic, mattingly_ergodicity_2002, wu2001large}. There is a substantial amount of works in the literature for
    studying the exponential convergence of the underdamped Langevin
    dynamics.  Below, we shall categorize them based on the norms and
    approaches to characterize the convergence.
\begin{enumerate}[leftmargin=*, wide, ]
\item[(\romannumeral1)] {(Convergence in $H^1(\stationary)$ norm)}.
  The exponential convergence of the kinetic Fokker-Planck equation in
  $H^1(\stationary)$ was proved by Villani in \cite[Theorem  35]{villani_hypocoercivity_2009}, which was inspired by early works of \cite{herau_isotropic_2004, nier2005hypoelliptic}. See also
  \cite{villani_hypocoercive_2007} for a brief overview of main
  ideas. The earlier work of \cite{mouhot2006quantitative} proved similar results on the torus without forcing term. Since $L^2(\stationary)$ norm is controlled by
  $H^1(\stationary)$ norm, this result automatically implies the
  convergence of \eqref{eqn::opL} in $L^2(\stationary)$. However, the
  decay rate therein is quite implicit; see
  \cite[Sec. 7.2]{villani_hypocoercivity_2009}. This approach is extended in \cite{baudoin2021gamma} to possibly singular potentials with convergence rates given in certain cases. 
  
  \medskip


\item[(\romannumeral2)] {(Convergence in a modified
    $L^2(\stationary)$ norm).} A more direct approach for convergence
  in $L^2(\stationary)$ was developed by Dolbeault, Mouhot and
  Schmeiser in
  \cite{dolbeault_hypocoercivity_2009,dolbeault_hypocoercivity_2015}, see also earlier ideas in   \cite{herau_hypocoercivity_2006}.
  They identified a modified $L^2(\stationary)$ norm, denoted by
  $\lyap$, such that $\lyap(\rho(t, x, v)) \rightarrow 0$
  exponentially fast for $\rho(t, \cdot, \cdot)$ evolving according to
  \eqref{eqn::fp}. 
    This hypocoercivity method was revisited and adapted in
  \cite{dolbeault2013exponential, grothaus2016hilbert, roussel2018spectral} to deal with the backward
  Kolmogorov equation \eqref{eqn::opL}, \ie, to show that
  $\lyap(f(t, \cdot, \cdot))$ decays to zero exponentially fast.  In
  Appendix \ref{subsec::comparison_hypocoercivity}, we will briefly
  revisit how to choose the Lyapunov function $\lyap$, based on
  \cite[Sec. 2]{dalalyan2018sampling}, because their setup is
  consistent with our $L^2(\stationary)$ estimate in
  \secref{subsec::main_result} above. We would like to remark that while \cite{roussel2018spectral} gets some rate, for which the scalings in $d$ and $\gamma$ are known, it is difficult to determine the optimal $\gamma$ for their convergence rate estimates.
  
  \smallskip
		
As a remark, the DMS method \cite{dolbeault_hypocoercivity_2009,dolbeault_hypocoercivity_2015} has been extended or adapted to study the convergence of spherical velocity Langevin equation \cite{grothaus_hypocoercivity_2014}, non-equilibrium Langevin dynamics 
\cite{iacobucci_convergence_2019}, Langevin dynamics with general
kinetic energy \cite{stoltz_langevin_2018}, temperature-accelerated
molecular dynamics \cite{stoltz_longtime_2018}, adaptive Langevin dynamics \cite{leimkuhler2020hypocoercivity}, dynamics with Boltzmann-type dissipation \cite{andrieu2021hypocoercivity}, dynamics with singular potentials \cite{camrud2021weighted}, just to name a few. It
might be interesting to study whether the variational framework
\cite{armstrong2019variational} we based on can be extended to these
cases.

\medskip
		
\item[(\romannumeral3)] {(Convergence in Wasserstein distance).}
		Baudoin discussed a  general framework of the \be{}  methodology \cite{bakry1985diffusions} to hypoelliptic and hypocoercive operators, based on which the exponential convergence of the kinetic Fokker-Planck equation (quantified by a Wasserstein distance associated with a special metric) was proved under certain assumptions on the potential $U(x)$ \cite[Theorem 2.6]{baudoin_wasserstein_2016}; see also \cite{baudoin_bakry-emery_2013}.
		
		\smallskip
		
		A different approach is the coupling method for underdamped
        Langevin dynamics \eqref{eqn::langevin}.  In
        \cite[Sec. 2]{dalalyan2018sampling}, for strongly convex
        potential $U$, Dalalyan and Riou-Durand considered the mixing
        of the marginal distribution in the $x$ coordinate, by a
        synchronous coupling argument; an estimate of the convergence
        rate was also explicitly provided, quantified by $W_2$
        distance \cite[Theorem 1]{dalalyan2018sampling}.  For more
        general potentials, Eberle, Guillin and Zimmer developed a
        hybrid coupling method, composed of synchronous and reflection
        couplings, to study the exponential convergence of probability
        distributions for the underdamped Langevin dynamics
        \eqref{eqn::langevin}, quantified by a Kantorovich semi-metric
        \cite{eberle_couplings_2019}. Unfortunately, their rates are dimension dependent in general.
        
        \medskip
        
        \item[(\romannumeral4)] (Convergence in relative entropy) Villani \cite{villani_hypocoercivity_2009} obtained exponential convergence of kinetic Fokker-Planck in the case of potentials with bounded Hessian, which is extended in \cite{baudoin_bakry-emery_2013}. A more quantitative convergence rate is obtained in \cite{ma2021there}. All of them essentially used Gamma calculus on a twisted metric so that derivatives in $x$ direction can be introduced. In \cite{cattiaux2019entropic}, exponential convergence of entropy is established for potentials that may not have bounded Hessians but satisfy a stronger weighted log-Sobolev inequality.

\end{enumerate}

There are other approaches to study the long time behavior of the
underdamped Langevin dynamics, \eg{}, Lyapunov function
\cite{talay2002stochastic, mattingly_ergodicity_2002, wu2001large, bakry2008rate} and spectral analysis
\cite{eckmann_spectral_2003,kozlov_effective_1989}. There are also works that extend the aforementioned approaches to dynamics with singular potentials \cite{conrad2010construction, cooke2011geometric,herzog2019ergodicity,lu2019geometric, baudoin2021gamma, camrud2021weighted}. We will not go
into details here.

\medskip

While our work is not the first one that studies the exponential convergence of underdamped Langevin dynamics, our estimates are more quantitative, and in certain cases, sharper than any existing result. In particular, for a large class of convex potentials, we establish an $O(\sqrt{m})$ convergence rate after optimizing in $\gamma$, which is independent of dimension and only assumes a mild upper bound (Assumption \ref{assump:hessian}) on the derivatives of the potential. To the best of our knowledge, this optimal $O(\sqrt{m})$ convergence rate is new in the literature. 

\medskip

Table~\ref{tab:rate} summarizes the previous results \cite{baudoin2021gamma, dalalyan2018sampling, ma2021there} under the assumption
$m\id \le \nabla_x^2 U \le L \id$ (and hence guarantee Assumptions
\ref{assump:poincare}-\ref{assump:spectral}) in the most interesting
regime $m\ll 1 \ll L$, with optimal choice of $\gamma$.
\begin{table}[htpb]
  \centering
 \begin{tabular}{c|c|c|c}
 \hline
     & \makecell{convergence rate \\ for arbitrary $\gamma$} & \makecell{convergence rate \\ with optimal $\gamma$} & criterion\\[.5mm] \hline\hline
   \cite[Corollary 3.19]{baudoin2021gamma}  & $O(\frac{m\gamma^3}{\gamma^4+L^2})$ & $O(\frac{m}{\sqrt{L}})$ & twisted $H^1$\\ \hline \cite{dalalyan2018sampling} & \makecell{only guarantees \\ convergence for $\gamma \ge \sqrt{L}$} & $O(\frac{m}{ \sqrt{L}})$ & $W_2$ \\\hline \makecell{\cite[Proposition 1]{ma2021there} \\ (after rescaling)} & \makecell{only guarantees \\ convergence for $\gamma \ge \sqrt{L}$} & $O(\frac{m}{ \sqrt{L}})$ & twisted KL \\ \hline Our work & $O(\frac{m\gamma}{m+\gamma^2})$ & $O(\sqrt{m})$ & $L^2$\\
   \hline
  \end{tabular}
  \medskip 
  \caption{Summary of the convergence rate $\nu$ depending on $d,m,L$ under the assumption
    $m\id \le \nabla_x^2 U \le L \id$ for the regime $m\ll 1 \ll L$.}
  \label{tab:rate}
\end{table}

\vspace{-0.2in}
To elaborate the comparison with result of \cite{ma2021there}, after a rescaling, they proved exponential convergence of \eqref{eqn::opL} with friction parameter (using their notations) $\gamma \sqrt{\xi}$ and convergence rate $O(\frac{\lambda}{\sqrt{\xi}})$, with constraints that requires (see \cite[Proof of Lemma 8]{ma2021there}) \[\left\{ \begin{aligned} &\frac{\xi}{2L}-(\frac{1}{4L}+\frac{1}{2m}) \lambda >0 \\ & \gamma(\frac{4\xi}{L}+1) -(\frac{1}{2m}+\frac{2}{L})\lambda >0 \\ & \frac{1}{2}- \frac{\xi}{2L}+(\frac{1}{4L}+\frac{1}{2m}) \lambda -\gamma(\frac{4\xi}{L}+1) +(\frac{1}{2m}+\frac{2}{L})\lambda \le 0.
\end{aligned} \right.\] Combined, these yield $\xi\ge O(L)$ and $\lambda \le O(m)$, which means the convergence rate cannot exceed $O(\frac{m}{\sqrt{L}})$. Moreover, they require $\gamma \ge O(1)$, or their friction parameter must be at least $O(\sqrt{L})$.

\medskip
	
We also comment that in the case where $\|\nabla_x^2 U\|\le L \mathrm{Id}$, but $U$ is not necessarily convex, our convergence rate is $\nu= O(\frac{m}{\sqrt{L}})$ after optimizing in $\gamma$ by choosing $\gamma\sim \sqrt{L}$, which matches the results of existing works \cite{baudoin2021gamma, ma2021there}. 

\smallskip
	\section{Proofs}
	\label{sec::proofs}
	
In this section, we present the statements and proofs of auxiliary lemmas, followed by the proofs of the two main theorems. Lemmas \ref{lem:mixedpoincare}, \ref{lem:phinablau} and \ref{lem:esthess} are the technical lemmas that prepare us for the elliptic regularity result in Lemma \ref{lem:ellreg}. The proof of the divergence Lemma, which builds up from elliptic regularity, is presented in Lemma \ref{lem:truetestfn}. The proof of Theorem \ref{thm:poincare} is then possible with the test functions obtained from Lemma \ref{lem:truetestfn}. Finally we present the proof of Theorem \ref{thm::decayrate} which follows from Theorem \ref{thm:poincare} and energy estimate. 

\medskip
	
    We start with the Poincar\'e inquality on tensorized space $(I\times\RR^d, \lambda\otimes \mu)$, which allows elliptic regularity to hold in the time-augmented state space. The proof is standard and thus omitted. 
	
	\begin{lemma}\label{lem:mixedpoincare}(Poincar{\'e} Inequality) For $f\in H^1(\lambda\otimes\mu)$,
		\begin{equation}\label{eqn:mixedpoincare}
		\| f-(f)_{\lambda \otimes \mu}\|_{L^2(\lambda \otimes\mu)}^2 \le \max\big\{ \frac{1}{m},\frac{T^2}{\pi^2}\big\}\Bigl(\|\partial_t f\|_{L^2(\lambda \otimes\mu)}^2+\|\nabla_x f\|_{L^2(\lambda \otimes\mu)}^2\Bigr). \end{equation}
	\end{lemma}

	The next lemma is also a technical lemma, the goal of which is to show that under Assumption \ref{assump:hessian}, $|\nabla^2 U|$ defines a bounded operator $H^1(\lambda\otimes\mu)\rightarrow L^2(\lambda\otimes\mu)$, which allows us to improve the regularity $u\in H^2(\lambda\otimes\mu)$ for $u$ being the solution of \eqref{eqn:mixedelliptic} in the proof of Lemma \ref{lem:ellreg}.
	\begin{lemma}(\cite[Lemma A.24]{villani_hypocoercivity_2009}) \label{lem:phinablau} For any $\phi \in H^1(\lambda\otimes\mu)$, we have
		\begin{equation}\label{eqn:phinablau}
		    \|\phi\nabla_x U\|_{L^2(\lambda \otimes\mu)}^2 \le 16\|\nabla_x \phi \|_{L^2(\lambda \otimes\mu)}^2+4Md\| \phi \|_{L^2(\lambda \otimes\mu)}^2,
		\end{equation}
		where $M$ is the constant in \eqref{eqn:stoltzcond9}.

	\end{lemma}
	\begin{proof}

		\begin{align*}
			\|\phi\nabla_x U\|_{L^2(\lambda \otimes\mu)}^2 
			&= \int_{I\times\mathbb{R}^d} \phi^2 \nabla_x U\cdot \nabla_x U \ud \lambda(t)\ud\mu(x)\\
			&= \int_{I\times\mathbb{R}^d} \nabla_x \cdot (\phi^2 \nabla_x U)\ud \lambda(t)\ud\mu(x)\\
			& = 2\int_{I\times\mathbb{R}^d} \phi \nabla_x \phi \cdot \nabla_x U\ud \lambda(t)\ud\mu(x) +\int_{I\times\mathbb{R}^d} \phi^2 \Delta_x U \ud \lambda(t)\ud\mu(x) \\
			& \stackrel{\eqref{eqn:stoltzcond9}}{\le}   \dfrac{1}{4} \| \phi \nabla_x U \|_{L^2(\lambda \otimes\mu)}^2 + 4 \|\nabla_x \phi \|_{L^2(\lambda \otimes\mu)}^2\\ & \qquad  + Md \|\phi\|_{L^2(\lambda \otimes\mu)}^2 +\frac{\delta}{2}\int_{I\times\RR^d} \phi^2 |\nabla_x U|^2 \ud \lambda(t)\ud\mu(x).
		\end{align*}
		We thus finish the proof of \eqref{eqn:phinablau} after rearranging and using $\delta<1$.
	\end{proof}

The following is a technical lemma that prepares us for the (mixed space-time) $H^2$ estimates of $u$, the solution of the elliptic equation \eqref{eqn:mixedelliptic}. This is a generalization of a similar $L^2$-$H^2$ regularity estimate in \cite[Proposition 5]{dolbeault_hypocoercivity_2015}, where only the spatial variable is considered, but our estimates are algebraically simpler thanks to Bochner's formula.  
Let us remark that we adopt the same scaling of parameters as \cite[Lemma 3.6]{bernard2022hypocoercivity}, especially in the most general case (iii).
\begin{lemma}\label{lem:esthess}
    For any $u\in H^2(\lambda\otimes\mu)$ such that $\wb{\nabla} u\in H_0^1(\lambda \otimes \mu)^{d+1}$, \begin{equation}\label{eqn:uH2est}
        \|D^2 u\|_{L^2(\lambda \otimes\mu)}^2=\sum_{i,j=0}^d\|\partial_{x_i}\partial_{x_j} u\|_{L^2(\lambda \otimes\mu)}^2 \le C\Bigl( \|\Lap u\|_{L^2(\lambda \otimes\mu)}^2+R^2\|\nabla_x u\|_{L^2(\lambda \otimes\mu)}^2\Bigr),
    \end{equation}
    Similarly,
    \begin{equation}\label{eqn:uH2estxonly}
        \|\nabla_x^2 u\|_{L^2(\lambda \otimes\mu)}^2\le C\Bigl( \|\nabla_x^*\nabla_x u\|_{L^2(\lambda \otimes\mu)}^2+R^2\|\nabla_x u\|_{L^2(\lambda \otimes\mu)}^2\Bigr).
    \end{equation}
    Here $C$ is a universal constant whose the precise value can be traced in the proof under different assumptions in Theorem \ref{thm::decayrate}, and $R$ is defined in Theorem~\ref{thm::decayrate}.
\end{lemma}
\begin{proof}
	We only prove \eqref{eqn:uH2est} since the proof of \eqref{eqn:uH2estxonly} follows from a similar argument. The starting point of the proof is Bochner's formula \begin{equation*}
			\sum_{i,j=0}^d |\partial_{x_i, x_j} u|^2=\wb{\nabla}u\cdot\wb{\nabla}\Lap u-(\nabla_xu)^{\top}\nabla_x^2U\nabla_xu -\Lap\dfrac{\abs{\wb{\nabla} u}^2}{2}.
		\end{equation*}
		Integrate over $\lambda \otimes \mu$ and (noticing the last term above has integral zero) we get \begin{equation}\label{eqn:bochner}
				\sum_{i,j=0}^{d}\|\partial_{x_i, x_j} u\|_{L^2(\lambda \otimes\mu)}^2 =\|\Lap u\|_{L^2(\lambda \otimes\mu)}^2-\int_{I\times\mathbb{R}^d} (\nabla_xu)^{\top} \nabla_x^2U \nabla_xu \ud \lambda(t)\ud\mu(x).
			\end{equation}
		This already verifies the conclusion in cases (i) (setting $K=0$) and (ii) with $C=1$. 
	
		\smallskip 
		
		Now we deal with the more general case, without assuming (\ref{eqn:lowerbdriccicurv}). Using \eqref{eqn:phinablau} with $\phi=\partial_{x_i} u,\ i=1,\cdots,d$, \begin{align*}
			& \int_{I\times \mathbb{R}^d} |\nabla_x u|^2|\nabla_x U|^2 \ud \lambda(t)\ud\mu(x) = \sum_{i=1}^d \int_{I\times \mathbb{R}^d} (\partial_{x_i} u)^2|\nabla_x U|^2 \ud \lambda(t)\ud\mu(x)  \\ & \qquad \stackrel{\eqref{eqn:phinablau}}{\le} 16
			\|D_x^2 u\|_{L^2(\lambda \otimes\mu)}^2 +4Md \int_{I\times \mathbb{R}^d} |\nabla_x u|^2 \ud \lambda(t)\ud\mu(x) \\ & \qquad  \stackrel{\eqref{eqn:bochner}}{=} 16\|\Lap u\|_{L^2(\lambda \otimes\mu)}^2 +4Md \int_{I\times \mathbb{R}^d} |\nabla_x u|^2 \ud \lambda(t)\ud\mu(x) \\ & \qquad \qquad \qquad  -16\int_{I\times\mathbb{R}^d} (\nabla_xu)^{\top} \nabla_x^2U \nabla_xu \ud \lambda(t)\ud\mu(x) \\ & \qquad \stackrel{\eqref{eqn:stoltzcond9}}{\le} 16\|\Lap u\|_{L^2(\lambda \otimes\mu)}^2+4 Md \int_{I\times \mathbb{R}^d} |\nabla_x u|^2 \ud \lambda(t)\ud\mu(x) \\ & \qquad \qquad \qquad +16M\int_{I\times \mathbb{R}^d} |\nabla_x u|^2(\sqrt{d}+|\nabla_x U|) \ud \lambda(t)\ud\mu(x)  \\ & \qquad \leftstackrel{d\ge 1}{\le} 16 \|\Lap u\|_{L^2(\lambda \otimes\mu)}^2 +20Md \int_{I\times \mathbb{R}^d} |\nabla_x u|^2 \ud \lambda(t)\ud\mu(x) \\ & \qquad \qquad\qquad +128M^2 \int_{I\times \mathbb{R}^d} |\nabla_x u|^2 \ud \lambda(t)\ud\mu(x)  +\dfrac{1}{2}\int_{I\times \mathbb{R}^d} |\nabla_x u|^2|\nabla_x U|^2 \ud \lambda(t)\ud\mu(x). \end{align*}
		Rearranging the terms, we arrive at  \begin{equation}\label{eqn:d2ud2U}\begin{aligned}\int_{I\times \mathbb{R}^d} |\nabla_x u|^2|\nabla_x U|^2 \ud \lambda(t)\ud\mu(x) \le 32\|\Lap u\|_{L^2(\lambda \otimes\mu)}^2 + (40Md+256M^2)\int_{I\times \mathbb{R}^d} |\nabla_x u|^2 \ud \lambda(t)\ud\mu(x) . \end{aligned}\end{equation}
		Therefore by \eqref{eqn:d2ud2U} and triangle inequality,\begin{align*}
			\|D^2 u\|_{L^2(\lambda \otimes\mu)}^2 & \leftstackrel{\eqref{eqn:stoltzcond9},\eqref{eqn:bochner}}{\le} \|\Lap u\|_{L^2(\lambda \otimes\mu)}^2 +M\int_{I\times \mathbb{R}^d} |\nabla_x u|^2(\sqrt{d}+|\nabla_x U|) \ud \lambda(t)\ud\mu(x) \\ & \le \;  \; \|\Lap u\|_{L^2(\lambda \otimes\mu)}^2 +M\sqrt{d}\|\nabla_x u\|_{L^2(\lambda \otimes\mu)}^2   +M\|\nabla_x u\|_{L^2(\lambda \otimes\mu)}\||\nabla_x u||\nabla_x U|\|_{L^2(\lambda \otimes\mu)}\\ & \leftstackrel{\eqref{eqn:d2ud2U}}{\le}  \|\Lap u\|_{L^2(\lambda \otimes\mu)}^2 +M\sqrt{d}\|\nabla_x u\|_{L^2(\lambda \otimes\mu)}^2  \\ & \qquad +M \|\nabla_x u\|_{L^2(\lambda \otimes\mu)}\big(6\|\Lap u\|_{L^2(\lambda \otimes\mu)}+(16M+\sqrt{40Md})\|\nabla_x u\|_{L^2(\lambda \otimes\mu)}\big) \\ & \le 4\|\Lap u\|_{L^2(\lambda \otimes\mu)}^2 +(19M^2+M\sqrt{40Md})\|\nabla_x u\|_{L^2(\lambda \otimes\mu)}^2. \qedhere
		\end{align*}
\end{proof}
	
	One of the key lemmas of our proof is the following result on elliptic regularity on the space $(I\times \RR^d, \lambda\otimes\mu)$. The solution to such elliptic equation will play an important role in the proof of Lemma~\ref{lem:truetestfn}. 
	\begin{lemma}\label{lem:ellreg}
		Consider the following elliptic equation: 
		\begin{equation}\label{eqn:mixedelliptic} \left\{ \begin{aligned} & 
			\Lap u=h & \mbox{in} & \ I\times \mathbb{R}^d,\\ &  \partial_t u(t=0, \cdot)=\partial_t u(t=T,\cdot)=0 & \mbox{in} & \ \mathbb{R}^d. 
	\end{aligned}\right.
		\end{equation}
		Assume $h\in H^{-1}(\lambda\otimes\mu)$, and $(h)_{\lambda \otimes \mu}=0$. Define the function space $$V=\bigl\{u\in H^1(\lambda\otimes\mu)~:~  (u)_{\lambda\otimes\mu}=0 \bigr\}. $$ Then \begin{enumerate}
			\item[(\romannumeral1)] There exists a unique $u\in V$ which is a weak solution to (\ref{eqn:mixedelliptic}). More precisely, for any $v\in H^1(\lambda \otimes \mu)$, we have \begin{equation*}
			    \int_{I\times\mathbb{R}^d} (\partial_t u\partial_t v +\nabla_x u \cdot \nabla_x v)\ud \lambda(t)\ud\mu(x) = \int_{I\times\mathbb{R}^d} hv \ud \lambda(t)\ud\mu(x).
			\end{equation*} Moreover, when $h\in L^2(\lambda \otimes\mu)$, we have the estimate \begin{equation}\label{eqn:uH1est}
				\|\partial_t u\|_{L^2(\lambda \otimes\mu)}^2 + 	\|\nabla_x u\|_{L^2(\lambda \otimes\mu)}^2 \le \max\bigl\{ \frac{1}{m},\frac{T^2}{\pi^2}\bigr\} \|h\|_{L^2(\lambda \otimes\mu)}^2.
			\end{equation}
			\item[(\romannumeral2)] If $h\in L^2(\lambda \otimes\mu)$, then the solution $u$ to (\ref{eqn:mixedelliptic}) satisfies $u\in H^2(\lambda \otimes \mu)$.
		\end{enumerate}
	\end{lemma}
	\begin{remark}
	 One could in fact estimate $\|u\|_{H^1(\lambda\otimes\mu)}$ using only $\|h\|_{H^{-1}(\lambda\otimes\mu)}$, but with a slightly worsened constant $\max\{\frac{1}{m},\frac{T^2}{\pi^2},1\}$ on the rhs. Since in our applications we only use $\|h\|_{L^2(\lambda \otimes\mu)}$, we opt for the current version of \eqref{eqn:uH1est} for simplicity.
	\end{remark}
	\begin{proof}
	    (\romannumeral1) $V$ is a linear Hilbert space and has non-zero elements (any function constant in $t$, and $H^1$ and mean zero in $x$ is included in $V$). Moreover, $V$ is a subspace of $H^1(\lambda\otimes\mu)$, and for the rest of the paper we equip it with the $H^1(\lambda\otimes\mu)$ norm. We also define the following inner-product: \[B(u,v):=\int_{I\times\mathbb{R}^d} (\partial_t u\partial_t v +\nabla_x u \cdot \nabla_x v)\ud \lambda(t)\ud\mu(x).\] One can easily verify $B(\cdot,\cdot)$ is an inner product on $V$. Notice that if $B(u,u)=0$ then $\partial_t u=\nabla_x u=0$, leaving $u$ to be a constant, which has to be $0$ since $(u)_{\lambda\otimes\mu}=0$. If $u$ is a weak solution of (\ref{eqn:mixedelliptic}), then for any $v\in V$, $B(u,v)=\int_{I\times \mathbb{R}^d} hv\ud \lambda(t)\ud\mu(x) $, and necessarily $(h)_{\lambda\otimes\mu}=0$ when we take $v=1$. 
		
		Since $(u)_{\lambda\otimes\mu}=0$, by Poincar\'{e} inequality (Lemma \ref{lem:mixedpoincare}) we can show $B$ is coercive under $H^1(\lambda\otimes\mu)$ norm in the sense of \begin{align*}
		    B[u,u] & = \|\partial_t u\|_{L^2(\lambda \otimes\mu)}^2+ \|\nabla_x u\|_{L^2(\lambda \otimes\mu)}^2 \\ & \ge \dfrac{1}{C}(\|\partial_t u\|_{L^2(\lambda \otimes\mu)}^2+\|\nabla_x u\|_{L^2(\lambda \otimes\mu)}^2+\|u\|_{L^2(\lambda \otimes\mu)}^2) \\ &= \dfrac{1}{C}\|u\|_{H^1(\lambda\otimes\mu)}^2.
		\end{align*}We can also show $B$ is bounded above since it is an inner-product and $B[u,u]\le \|u\|_{H^1(\lambda\otimes\mu)}^2$. Define a linear functional on $V$: $H(v):=\int_{I\times\mathbb{R}^d} hv \ud \lambda(t)\ud\mu(x)$. One can verify the boundedness of $H$: $$|H(v)|\le \|h\|_{H^{-1}(\lambda\otimes\mu)}\|v\|_{H^1(\lambda\otimes\mu)}. $$  Thus by Lax-Milgram's Theorem, the equation (\ref{eqn:mixedelliptic}) has a unique weak solution $u\in V$. Moreover,
		\begin{align*}
			( \|\partial_t u\|_{L^2(\lambda \otimes\mu)}^2 & + 	\|\nabla_x u\|_{L^2(\lambda \otimes\mu)}^2 )^2= B[u,u]^2 \\ & = \Bigl(\int_{I\times \mathbb{R}^d} hu\ud \lambda(t)\ud\mu(x)\Bigr)^2  \le \|h\|_{L^2(\lambda \otimes\mu)}^2\|u\|_{L^2(\lambda \otimes\mu)}^2 \\ & \leftstackrel{\eqref{eqn:mixedpoincare}}{\le} \max\big\{ \frac{1}{m},\frac{T^2}{\pi^2} \big\}\|h\|_{L^2(\lambda \otimes\mu)}^2 \bigl(	\|\partial_t u\|_{L^2(\lambda \otimes\mu)}^2 + 	\|\nabla_x u\|_{L^2(\lambda \otimes\mu)}^2\bigr),
		\end{align*}
		and the desired estimate follows. 
		
		\medskip 
		
	(\romannumeral2) For each $i=1,2,\cdots,d$, consider the elliptic equation 
	\begin{equation}\label{eqn:highordereqn}\left\{ \begin{aligned}
				& \Lap w_i=\partial_{x_i}h-\nabla_x u\cdot \nabla_x \partial_{x_i} U & \mbox{in} & \ I\times \mathbb{R}^d,\\ & \partial_t w_i(t=0, \cdot)=\partial_t w_i(t=T,\cdot)=0 & \mbox{in} & \ \mathbb{R}^d. \end{aligned} \right.
		\end{equation}
		The motivation of considering \eqref{eqn:highordereqn} is that, if we formally differentiate \eqref{eqn:mixedelliptic} with respect to $\partial_{x_i}$, then $\partial_{x_i} u $ satisfies precisely the equation \eqref{eqn:highordereqn} for $w_i$. Hence, our plan is to use part (i) to establish $w_i\in H^1(\lambda\otimes\mu)$, then argue that $w_i-\partial_{x_i} u$ must be constant.
		
		\smallskip 
		
		We first verify the rhs of \eqref{eqn:highordereqn} has total integral zero. Indeed 
		\begin{align*}
			& \int_{I\times\mathbb{R}^d} (\partial_{x_i} h-\nabla_xu \cdot \nabla_x \partial_{x_i} U) \ud \lambda(t)\ud\mu(x) \\ & \qquad = \int_{I\times \mathbb{R}^d} (h\partial_{x_i} U-\nabla_xu \cdot \nabla_x \partial_{x_i} U) \ud \lambda(t)\ud\mu(x) \\ & \qquad = \int_{I\times \mathbb{R}^d} \big(\Lap u\partial_{x_i} U-\nabla_xu \cdot \nabla_x \partial_{x_i} U\big) \ud \lambda(t)\ud\mu(x) \\ & \qquad  = \int_{I\times \mathbb{R}^d} \big(\partial_{t}u\partial_{tx_i}U+\nabla_x u\cdot \nabla_x\partial_{x_i} U-\nabla_xu \cdot \nabla_x \partial_{x_i} U\big) \ud \lambda(t)\ud\mu(x) = 0.
		\end{align*}
			The next step is to show rhs is in $H^{-1}(\lambda\otimes\mu)$. Pick a test function $\phi\in H^1(\lambda\otimes\mu)$ with $\|\phi\|_{H^1(\lambda\otimes\mu)}=1$, and by Lemma \ref{lem:phinablau}: \begin{align*}
			&  \int_{I\times\mathbb{R}^d} (\partial_{x_i} h-\nabla_x u\cdot\nabla_x\partial_{x_i} U)\phi \ud \lambda(t)\ud\mu(x) \\ & \qquad \le \int_{I\times\mathbb{R}^d} (-h\partial_{x_i} \phi +h\phi \partial_{x_i} U) \ud \lambda(t)\ud\mu(x)+\int_{I\times\mathbb{R}^d}|\phi \nabla_x u||\nabla_x\partial_{x_i} U| \ud \lambda(t)\ud\mu (x) \\ & \qquad \stackrel{\eqref{eqn:stoltzcond9}}{\le}  \|h\|_{L^2(\lambda \otimes\mu)}(1+\|\phi\partial_{x_i} U\|_{L^2(\lambda \otimes\mu)})+M\int_{I\times\mathbb{R}^d}|\phi \nabla_x u|(\sqrt{d}+|\nabla_x U|) \ud \lambda(t)\ud\mu (x)\\ & \qquad \le \|h\|_{L^2(\lambda \otimes\mu)}(1+\|\phi\partial_{x_i} U\|_{L^2(\lambda \otimes\mu)})+M\|\nabla_x u\|_{L^2(\lambda \otimes\mu)}(\sqrt{d}+\|\phi\nabla_x U\|_{L^2(\lambda \otimes\mu)}) \\ & \qquad \leftstackrel{\eqref{eqn:phinablau},\eqref{eqn:uH1est}}{\le}  C(M,d)\|h\|_{L^2(\lambda\otimes\mu)},
		\end{align*}
		where $C(M,d)>0$ is a constant depending on $M,d$. Therefore, by $(\romannumeral1)$ we know there exists a $w_i\in V$ which is the weak solution of (\ref{eqn:highordereqn}). Finally, comparing \eqref{eqn:mixedelliptic} and \eqref{eqn:highordereqn}, we observe that $\Lap (w_i-\partial_{x_i}u)=0$ in the sense of distributions, which by (i) indicates $w_i-\partial_{x_i} u$ must be constant, which must be $-(\partial_{x_i} u)_{\lambda\otimes\mu}$, since by construction $w\in V$ and $(w)_{\lambda\otimes\mu}=0$. This also means $\partial_{x_i} u \in H^1(\lambda\otimes\mu)$ since $w_i\in H^1(\lambda\otimes\mu)$. We end the proof of $u\in H^2(\lambda\otimes\mu)$ by writing $\partial_{tt} u =\nabla_x^*\nabla_x u-h \in L^2(\lambda \otimes\mu) $.
	\end{proof}
	
	We finally need a lemma for the solution of a divergence equation with Dirichlet boundary conditions. The resolution of divergence equation is an important tool in mathematical fluid dynamics (see the book \cite[Section III.3]{galdi2011introduction}). However, in order to obtain more natural estimate on the constants, instead of resorting to the aforementioned Bogovskii's operator, we take advantage of the structure of space $L^2(\mu)$ by eigenspace decomposition, which is made possible thanks to Assumption \ref{assump:spectral}. This will provide us test functions which play a crucial role in the proof of Theorem~\ref{thm:poincare}.
	
	\begin{lemma}\label{lem:truetestfn}
   For any function $f\in L^2(\lambda\otimes\mu)$ with $(f)_{\lambda\otimes \mu}=0$, there exist two functions $\phi_0 \in H_0^1(\lambda\otimes \mu)$ and $\Phi \in H^2(\lambda\otimes\mu)$ such that $\nabla_x \Phi \in H_0^1(\lambda\otimes\mu)^d$ and  \begin{equation}\label{eqn:divergence}
      -\partial_t\phi_0+\nabla_x^*\nabla_x \Phi = f
   \end{equation} with estimates 
   \begin{equation}\label{eqn:estphi}
       \|\phi_0\|_{L^2(\lambda\otimes\mu)}+ \|\nabla_x \Phi\|_{L^2(\lambda\otimes\mu)} \le C\Big(\frac{1}{\sqrt{m}(1-e^{-\sqrt{m} T})}+T\Big)\| f\|_{L^2(\lambda\otimes\mu)}
   \end{equation}  and 
   \begin{equation}\label{eqn:estdiv}
       \|\nabla_x \phi_0\|_{L^2(\lambda\otimes\mu)} +  \|\wb{\nabla}\nabla_x \Phi\|_{L^2(\lambda\otimes\mu)}
       \le C\Bigl(1+RT+ \frac{1}{(1-e^{-\sqrt{m}T})^2}+ \frac{R}{\sqrt{m}(1-e^{-\sqrt{m}T})^2}\Bigr) \| f\|_{L^2(\lambda\otimes\mu)}.
   \end{equation} Here $C$ is a universal constant and $R$ is the constant defined in Theorem~\ref{thm::decayrate}.
\end{lemma}
\begin{remark}
    We believe the correct scaling of the rhs should be $O(\frac{1}{T})$ as $T\to 0$, which we are unable to obtain, due to the pessimistic estimates in the last two lines of \eqref{eqn:fL2norm} that changed the scaling of the last two terms from $O(1)$ to $O(T^2)$, but will not pursue further since in the proof of Theorem \ref{thm::decayrate} we only take $T=\frac{1}{\sqrt{m}}$. As we mentioned iearlier after Theorem \ref{thm:poincare}, the scaling of $O(\frac{1}{T})$ as $T\to 0$ should come from \eqref{eqn:psi2prest}.
\end{remark}
Before we proceed to the proof, let us give a brief heuristic argument on why we need to introduce the space of harmonic functions (i.e. the space $\mathbb{H}$ that appears at the beginning of the proof) and consider orthogonal projection on it. Indeed, a direct way to look for a solution of \eqref{eqn:divergence} is to look for that of \eqref{eqn:mixedelliptic} and set $\phi_0=\partial_t u, \Phi = u$. However, these test functions do not satisfy the appropriate boundary conditions. In particular, if solution of \eqref{eqn:mixedelliptic} satisfy $\nabla_x u(t=0,\cdot) = \nabla_x u(t=T,\cdot)=0$, then necessarily $f$ has to be perpendicular to the space of harmonic functions. Meanwhile, the harmonic part of $f$ requires special treatment from us and brings technical difficulty to the proof. However, thanks to Assumption \ref{assump:spectral}, one can decompose the harmonic part of $f$ using separation of variables, which enables us to obtain the solution of divergence equation by constructing it for each component and adding them up.  
\begin{proof}
Let $\mathbb{H}$ be the subspace of $L^2(\lambda\otimes\mu)$ that consists of ``harmonic functions'', in other words, $f\in \mathbb{H}$ if and only if $\Lap f=0$. We consider the decomposition $f=f^{(1)}+f^{(2)}$ where $f^{(1)}\in \mathbb{H}$ and $f^{(2)}\perp \mathbb{H}$. Since $1\in \mathbb{H}$ we know $(f^{(2)})_{\lambda\otimes \mu}=0$ and hence $(f^{(1)})_{\lambda\otimes \mu}=0$. Therefore by linearity it suffices to consider $f^{(1)}$ and $f^{(2)}$ separately. For $f^{(2)}$, the equation \begin{equation} \label{eqn:ufkappa}
		    \left\{ \begin{aligned}
				& \Lap u=f^{(2)} &\mbox{in} & \ I\times \mathbb{R}^d, \\ &\partial_t u(t=0, \cdot)=\partial_t u(t=T,\cdot)=0 & \mbox{in} & \ \mathbb{R}^d \end{aligned} \right.
		\end{equation}
    has a unique solution in $V\cap H^2(\lambda\otimes\mu)$ by Lemma~\ref{lem:ellreg}. Moreover, for any $v\in \mathbb{H} \cap H^2(\lambda\otimes\mu)$, integration by parts yields \begin{multline*}
        0=\int_{I\times \RR^d} f^{(2)}v \ud \lambda(t) \ud \mu(x) = B[u,v] \\ = \int_{I\times \RR^d} u\Lap v \ud \lambda(t) \ud \mu(x) + \int_{\RR^d} \big(u(T)\partial_t v(T) - u(0)\partial_t v(0)\big) \ud\mu(x)
    \end{multline*} 
    Therefore, since $v$ is arbitrary, we have $u(T)=u(0)=0$, which implies $\nabla_x u\in H_0^1(\lambda\otimes \mu)^d$. Also by construction of boundary conditions $\partial_t u\in H_0^1(\lambda\otimes \mu)$. Thus for $f^{(2)}$ part, it suffices to take correspondingly $\phi_0^{(2)}=\partial_t u,~\Phi^{(2)}= u$ with the estimates \begin{equation}\label{eqn:perpphi}
         \| \wb{\nabla}u\|_{L^2(\lambda\otimes\mu)}^2 \leftstackrel{\eqref{eqn:uH1est}}{\le} C\max\big\{ \frac{1}{m},T^2 \big\} \|f^{(2)}\|^2_{L^2(\lambda\otimes\mu)},
    \end{equation} and \begin{equation}\label{eqn:perpgradphi}\| D^2 u\|_{L^2(\lambda\otimes\mu)}^2 \leftstackrel{\eqref{eqn:uH2est},\eqref{eqn:perpphi}}{\le} C(1+\dfrac{R^2}{m}+R^2T^2)\|f^{(2)}\|_{L^2(\lambda\otimes\mu)}^2. \end{equation}
    
   We now consider the $f^{(1)}$ part. Since $
\{1\}\cup \{w_\alpha\} $ forms an orthonormal basis in $L^2(\mu)$ and $(f^{(1)})_{\lambda\otimes \mu}=0$, we have an orthogonal decomposition \begin{equation*}
        f^{(1)}(t, x)=f_0(t)+\sum_{\alpha} f_\alpha(t) w_\alpha (x).
    \end{equation*}
    Since $f^{(1)}$ is harmonic, \begin{equation*}
        0=\Lap f^{(1)}= -f_0''(t)+\sum_{\alpha} (-f''_\alpha(t)+\alpha^2f_\alpha(t)) w_\alpha (x)
    \end{equation*} and therefore $f_0(t)$ is an affine function $f_0(t)=c_0(t-\frac{T}{2})$ for some constant $c_0$, as $f_0(t)$ has integral zero. Moreover for $\alpha>0$ there exist constants $c_\pm^\alpha$  such that \begin{equation*}
        f_\alpha(t)=c_+^\alpha e^{-\alpha t}+c_-^\alpha e^{-\alpha (T-t)}.
    \end{equation*} 
    Therefore, by orthogonality in $L^2(\lambda \otimes\mu)$, we can write for some constant $C\in (1,\infty)$, \begin{align*}
        \|f\|_{L^2(\lambda\otimes\mu)}^2 & = \|f^{(2)}\|_{L^2(\lambda\otimes\mu)}^2 + c_0^2\|t-\frac{T}{2}\|_{L^2(\lambda)}^2+\sum_\alpha \|c_+^\alpha e^{-\alpha t} + c_-^\alpha e^{-\alpha(T-t)} \|_{L^2(\lambda)}^2 \\ & = \|f^{(2)}\|_{L^2(\lambda\otimes\mu)}^2 + \frac{T^2 c_0^2}{12}+\sum_\alpha \Big(\bigl((c_+^\alpha)^2+(c_-^\alpha)^2 \bigr)\frac{1-e^{-2\alpha T}}{2\alpha T} + 2c_+^\alpha c_-^\alpha e^{-\alpha T} \Big) \\ & \ge \|f^{(2)}\|_{L^2(\lambda\otimes\mu)}^2 + \frac{T^2 c_0^2}{12}+\sum_\alpha \bigl((c_+^\alpha)^2+(c_-^\alpha)^2 \bigr)\Big(\frac{1-e^{-2\alpha T}}{2\alpha T} - e^{-\alpha T} \Big) \\ & \ge \|f^{(2)}\|_{L^2(\lambda\otimes\mu)}^2 + \frac{T^2 c_0^2}{12} + \frac{1}{C}\sum_\alpha \bigl((c_+^\alpha)^2+(c_-^\alpha)^2 \bigr)\frac{(1-e^{-\alpha T})^3}{\alpha T}. \stepcounter{equation} \tag{\theequation}\label{eqn:fL2norm}
    \end{align*}

    The construction of test functions for $f_0(t)$ is straightforward: We simply take $\Phi^{(0)}=0$ and $\phi_0^{(0)}(t,x)=\frac{c_0}{2}(t^2-tT)$.    We then construct $\phi_{0,\alpha},\Phi_\alpha$ for each component of the sum $e^{-\alpha t} w_\alpha(x)$, and therefore the functions $\phi_{0,\alpha}(T-t,\cdot),\Phi_{\alpha}(T-t,\cdot)$ also apply to the component $e^{-\alpha (T-t)}w_\alpha(x)$, so that the eventual test functions $\phi_0,\Phi$ can be obtained after taking linear combination. The goal is to find $\phi_{0,\alpha},\Phi_\alpha$ such that \begin{equation*}
        -\partial_t\phi_{0,\alpha}+\nabla_x^*\nabla_x \Phi_\alpha= e^{-\alpha t}w_\alpha(x).
    \end{equation*} Since $w_\alpha\in H^2(\lambda\otimes\mu)$, in order to eliminate the $x$ part of the equation, we can take the natural ansatz by separation of variables $\phi_{0,\alpha}=\psi_{1,\alpha}(t)w_\alpha (x)$ and $\Phi_\alpha =\psi_{2,\alpha}(t)  w_\alpha(x)$, and the two functions $\psi_{1,\alpha}(t),\psi_{2,\alpha}(t)$ should satisfy $\psi_{1,\alpha}(0)=\psi_{1,\alpha}(T)=\psi_{2,\alpha}(0)=\psi_{2,\alpha}(T)=0$ as well as the equation \begin{equation}\label{eq:psieqn}
        -\psi_{1,\alpha}'(t)+\alpha^2\psi_{2,\alpha}(t)=e^{-\alpha t}.
    \end{equation} 
    Integrating \eqref{eq:psieqn} against $t$, we obtain the necessary and sufficient condition 
    \begin{equation}\label{eq:psieqnness}
        \int_0^T \psi_{2,\alpha}(t) \ud t = \frac{1-e^{-\alpha T}}{\alpha^3}.
    \end{equation}
    Of course there exists infinitely many possible solutions, since for any $\psi_{2,\alpha}$ that vanishes at both time boundaries and satisfies \eqref{eq:psieqnness}, the choice $\psi_{1,\alpha} = \int_0^t (\alpha^2\psi_{2,\alpha} (\tau) -e^{-\alpha \tau})\ud \tau$ also vanishes at both time boundaries. Therefore we only need to choose a particular one to satisfy the desired estimates. Let us introduce a short-hand notation $\ell=e^{-\alpha T} \in (0,1)$. Our idea is to find $\psi_{2,\alpha}$ of the form $\psi_{2,\alpha}(t)= \frac{1}{\alpha^2}g(e^{-\alpha t})$, which after a change of variable $s:=e^{-\alpha t}$ turns the condition \eqref{eq:psieqnness} into $\int_{\ell}^1 \frac{g(s)}{s}\ud s = 1-\ell$, and the boundary conditions into $g(1)=g(\ell)=0$. Hence, we may finish our construction by picking $g(s)=sh(s)$ with
    \begin{equation*}
    h(x)= \frac{6}{(1-\ell)^2}(x-\ell)(1-x).
    \end{equation*} 
    From the expression we can directly derive (using $\alpha\ge \sqrt{m}$) 
    \begin{equation*} 0 \le g(s)\le \frac{3}{2}s \ \mbox{ and } \ |g'(s)|\le \dfrac{4}{1-\ell}= \frac{4}{1-e^{-\alpha T}}. \end{equation*}
    One can explicitly compute 
    \begin{align*} & \|\psi_{2,\alpha}\|_{L^2(\lambda)}^2=\dfrac{1}{\alpha^4T}\int_0^T g(e^{-\alpha t})^2 \ud t = \dfrac{1}{\alpha^5T}\int_\ell^1 \dfrac{g(s)^2}{s}\ud s = \dfrac{3(1-e^{-2\alpha T})}{5\alpha^5T}, \stepcounter{equation} \tag{\theequation} \label{eqn:psi2est} \\ \mbox{and  } & \|\psi_{2,\alpha}'\|_{L^2(\lambda)}^2=\dfrac{1}{\alpha^2T}\int_0^T g'(e^{-\alpha t})^2 e^{-2\alpha t} \ud t = \dfrac{1}{\alpha^3T}\int_\ell^1 g'(s)^2 s \ud s \le \dfrac{8}{\alpha^3T(1-e^{-\alpha T})}.\stepcounter{equation} \tag{\theequation} \label{eqn:psi2prest} \end{align*} 
    Moreover since $\psi_{1,\alpha}'(t) = \alpha^2\psi_{2,\alpha}(t) - e^{-\alpha t}$ from \eqref{eq:psieqn},
    \begin{equation}  \label{eqn:psi1prest}\|\psi_{1,\alpha}'\|_{L^2(\lambda)}^2\le 2\alpha^4 \|\psi_{2,\alpha}\|_{L^2(\lambda)}^2 +\dfrac{1-e^{-2\alpha T}}{\alpha T}\le \dfrac{3(1-e^{-2\alpha T})}{\alpha T}.\end{equation} Finally since \begin{equation*}\psi_{1,\alpha}(t)=\int_0^t (g(e^{-\alpha s})-e^{-\alpha s})\ud s =\dfrac{1}{\alpha}\int_{e^{-\alpha t}}^1 (\dfrac{g(\tau)}{\tau}-1)\ud\tau  = \dfrac{1}{\alpha} r(e^{-\alpha t})\end{equation*} with \[r(s) = \int_s^1 (h(\tau)-1)\ud\tau = \frac{(s-\ell)(1-s)(1+\ell-2s)}{(1-\ell)^2},\] we can estimate \begin{equation}\label{eqn:psi1est}
         \alpha^2\|\psi_{1,\alpha}\|_{L^2(\lambda)}^2= \frac{1}{\alpha T} \int^1_\ell  \frac{r(t)^2}{t}\ud t = \frac{(1-\ell)^3}{\alpha T} \int_0^1 \frac{s^2(1-s)^2(1-2s)^2}{(1-\ell)s+\ell}\ud s \le \frac{C(1-e^{-\alpha T})^3}{\alpha T}.
    \end{equation}
     To sum up, our construction of test functions write
    \begin{align*}
    \phi_0 & = \partial_t u+c_0\frac{t^2-tT}{2} + \sum_\alpha(c_+^\alpha \psi_{1,\alpha}(t) + c_-^\alpha \psi_{1,\alpha}(T-t))w_\alpha(x), \\ \Phi & =  u + \sum_\alpha(c_+^\alpha \psi_{2,\alpha}(t) + c_-^\alpha \psi_{2,\alpha}(T-t))w_\alpha(x), \end{align*}  here we recall that $u$ is the solution of \eqref{eqn:ufkappa}.
    
We now establish the estimates by direct calculations, which is possible since the variables are separated. Notice that for $\alpha, \beta$, 
    \[ \langle \nabla_x w_\alpha, \nabla_x w_\beta\rangle_{L^2(\mu)}=  \langle w_\alpha,  \nabla_x^*\nabla_x w_\beta\rangle_{L^2(\mu)}= \beta^2 \langle w_\alpha,   w_\beta\rangle_{L^2(\mu)}=\alpha^2 \delta_{\alpha,\beta},\]
    hence cross terms in the expansion of $\|\sum_\alpha(c_+^\alpha \psi_{2,\alpha}(t) + c_-^\alpha \psi_{2,\alpha}(T-t))\nabla_x w_\alpha(x)\|_{L^2(\lambda\otimes\mu)}^2$ vanish. Therefore, we can estimate  
    \begin{align*}\stepcounter{equation} \tag{\theequation} \label{eqn:phiL2est}
      & \|\phi_0\|^2_{L^2(\lambda\otimes\mu)}  + \|\nabla_x\Phi\|^2_{L^2(\lambda\otimes\mu)}\\ & \le 3\Big(\|\partial_t u\|_{L^2(\lambda\otimes\mu)}^2 + \frac{c_0^2}{4}\|t^2-tT\|_{L^2(\lambda)}^2 + \sum_\alpha \|c_+^\alpha \psi_{1,\alpha}(t) + c_-^\alpha \psi_{1,\alpha}(T-t)\|_{L^2(\lambda)}^2\| w_\alpha\|_{L^2(\mu)}^2 \\ & \qquad + \|\nabla_x u\|_{L^2(\lambda\otimes\mu)}^2 + \|\sum_\alpha (c_+^\alpha \psi_{2,\alpha}(t) + c_-^\alpha \psi_{2,\alpha}(T-t))\nabla_x w_\alpha\|_{L^2(\lambda\otimes\mu)}^2  \Big)  \\ & \leftstackrel{\eqref{eqn:uH1est} }{\le} 6\Bigl(\max\{\frac{1}{m},T^2 \} \|f^{(2)}\|_{L^2(\lambda \otimes\mu)}^2 + \frac{c_0^2T^4}{120} + \sum_\alpha ((c_+^\alpha)^2+(c_-^\alpha)^2) \|\psi_{1,\alpha}\|_{L^2(\lambda)}^2  \\ & \qquad +\sum_{\alpha} \| c_+^\alpha \psi_{2,\alpha}(t) + c_-^\alpha \psi_{2,\alpha}(T-t)\|_{L^2(\lambda)}^2 \|\nabla_x w_\alpha\|_{L^2(\mu)}^2 \Bigr) \\ & \le C\Big(\max\{\frac{1}{m},T^2\}\|f^{(2)}\|_{L^2(\lambda \otimes\mu)}^2 + c_0^2T^4 + \sum_\alpha ((c_+^\alpha)^2+(c_-^\alpha)^2) (\|\psi_{1,\alpha}\|_{L^2(\lambda)}^2+\alpha^2\|\psi_{2,\alpha}\|_{L^2(\lambda)}^2) \Big) \\ & \leftstackrel{\eqref{eqn:psi1est},\eqref{eqn:psi2est}}{\le} C\Big(\max\{\frac{1}{m},T^2\}\|f^{(2)}\|_{L^2(\lambda \otimes\mu)}^2 + c_0^2T^4 + \sum_\alpha \frac{1}{\alpha^2}((c_+^\alpha)^2+(c_-^\alpha)^2) \frac{(1-e^{-\alpha T})^3+1-e^{-2\alpha T}}{\alpha T} \Big) \\ & \leftstackrel{\eqref{eqn:fL2norm}}{\le} C\max\Big\{\frac{1}{m(1-e^{-\sqrt{m} T})^2},T^2\Big\} \|f\|_{L^2(\lambda\otimes\mu)}^2.
    \end{align*} 
   Here in the last line when we used \eqref{eqn:fL2norm}, the worse factor $(1-e^{-\sqrt{m}T})^{-2}$ comes only from the last term on the line above. This establishes \eqref{eqn:estphi}.
   Using similar arguments, we can estimate \begin{align*}\stepcounter{equation} \tag{\theequation} \label{eqn:dbarphi0}
        \|\nabla_x \phi_0\|_{L^2(\lambda \otimes\mu)}^2  & =  \Big\| \nabla_x\partial_t u+ \sum_\alpha(c_+^\alpha \psi_{1,\alpha}(t) - c_-^\alpha \psi_{1,\alpha}(T-t))\nabla_x w_\alpha(x)\Big\|^2_{L^2(\lambda \otimes\mu)} \\ & \le 2\Big( \|\nabla_x\partial_t u \|_{L^2(\lambda \otimes\mu)}^2  + \sum_\alpha \alpha^2\|c_+^\alpha \psi_{1,\alpha}(t) + c_-^\alpha \psi_{1,\alpha}(T-t)\|_{L^2(\lambda)}^2  \Big) \\ & \le C\Big( \|\nabla_x\partial_t u \|_{L^2(\lambda \otimes\mu)}^2 + \sum_\alpha((c_+^\alpha)^2+ (c_-^\alpha)^2)  \alpha^2 \|\psi_{1,\alpha}\|_{L^2(\lambda)}^2 \Big)
        \\ & \leftstackrel{\eqref{eqn:psi1est}}{\le} C\Big( \|\nabla_x\partial_t u \|_{L^2(\lambda \otimes\mu)}^2 + \sum_\alpha((c_+^\alpha)^2+ (c_-^\alpha)^2) \frac{(1-e^{-\alpha T})^3}{\alpha T} \Big), 
    \end{align*}
        as well as 
    \begin{align*}\stepcounter{equation} \tag{\theequation} \label{eqn:dtphii}
        \|\partial_t \nabla_x\Phi\|_{L^2(\lambda \otimes\mu)}^2 & =  \Bigl\|\nabla_x \partial_t u + \sum_\alpha(c_+^\alpha \psi_{2,\alpha}'(t) - c_-^\alpha \psi_{2,\alpha}'(T-t))\nabla_x w_\alpha(x)\Bigr\|_{L^2(\lambda\otimes\mu)}^2 \\ & \le 2 \Big( \| \nabla_x \partial_t u\|_{L^2(\lambda \otimes\mu)}^2 + \sum_\alpha \|c_+^\alpha \psi_{2,\alpha}'(t) - c_-^\alpha \psi_{2,\alpha}'(T-t)\|_{L^2(\lambda)}^2\|\nabla_x w_\alpha\|_{L^2(\mu)}^2\Big) \\ & \le C \Big( \| \nabla_x \partial_t u\|_{L^2(\lambda \otimes\mu)}^2 + \sum_\alpha \alpha^2((c_+^\alpha)^2+ (c_-^\alpha)^2)\|\psi_{2,\alpha}'\|_{L^2(\lambda)}^2\Big) \\ & \leftstackrel{\eqref{eqn:psi2prest}}{\le} C\Big( \| \nabla_x \partial_t u\|_{L^2(\lambda \otimes\mu)}^2 + \sum_\alpha ((c_+^\alpha)^2+ (c_-^\alpha)^2)\frac{1}{\alpha T(1-e^{-\alpha T})}\Big). 
    \end{align*}
    We finally treat the terms from $\nabla^2_x \Phi$:
    \begin{align*} \stepcounter{equation} \tag{\theequation} \label{eqn:dxxPhi}
         \| \nabla^2_x \Phi\|_{L^2(\lambda \otimes\mu)}^2 &  \leftstackrel{\eqref{eqn:uH2estxonly}}{\le} C\Big(\|\nabla_x^*\nabla_x \Phi\|_{L^2(\lambda \otimes\mu)}^2 +R^2 \|\nabla_x \Phi\|_{L^2(\lambda\otimes\mu)}^2\Big) \\ & \leftstackrel{\eqref{eqn:divergence},\eqref{eqn:phiL2est}}{\le} C\Big(\Big\|f+\partial_{tt} u+c_0(t-\frac{T}{2}) + \sum_\alpha(c_+^\alpha \psi_{1,\alpha}'(t) - c_-^\alpha \psi_{1,\alpha}'(T-t))w_\alpha(x)\Big\|_{L^2(\lambda \otimes\mu)}^2\\ & \qquad  \qquad +R^2\big(T^2+ \frac{1}{m(1-e^{-\sqrt{m}T})^2}\big) \|f\|_{L^2(\lambda \otimes\mu)}^2\Big)\\ & \le C\Big(\|\partial_{tt}u\|_{L^2(\lambda \otimes\mu)}^2+c_0^2T^2+ \sum_\alpha ((c_+^\alpha)^2+ (c_-^\alpha)^2)\|\psi_{1,\alpha}'\|_{L^2(\lambda)}^2 \\ & \qquad +\bigl(1+R^2T^2+ \frac{R^2}{m(1-e^{-\sqrt{m}T})^2}\bigr) \|f\|_{L^2(\lambda \otimes\mu)}^2\Big) \\ & \leftstackrel{\eqref{eqn:psi1prest}}{\le} C\Big(\|\partial_{tt}u\|_{L^2(\lambda \otimes\mu)}^2+c_0^2T^2+ \sum_\alpha ((c_+^\alpha)^2+ (c_-^\alpha)^2)\frac{1-e^{-2\alpha T}}{\alpha T} \\ & \qquad +\bigl(1+R^2T^2+ \frac{R^2}{m(1-e^{-\sqrt{m}T})^2}\bigr) \|f\|_{L^2(\lambda \otimes\mu)}^2\Big).
    \end{align*}
    Adding together \eqref{eqn:dbarphi0},\eqref{eqn:dtphii},\eqref{eqn:dxxPhi}, and we arrive at \begin{align*}
         \|\nabla_x \phi_0\|_{L^2(\lambda\otimes\mu)}^2   +  \|\wb{\nabla}\nabla_x \Phi\|_{L^2(\lambda\otimes\mu)}^2
        & \le C\Bigl(\|D^2 u\|_{L^2(\lambda \otimes\mu)}^2 + c_0^2 T^2 +\sum_\alpha ((c_+^\alpha)^2+ (c_-^\alpha)^2)\frac{1}{\alpha T(1-e^{-\alpha T})} +   \\ & \qquad \bigl(1+R^2T^2+ \frac{R^2}{m(1-e^{-\sqrt{m}T})^2}\bigr) \|f\|_{L^2(\lambda \otimes\mu)}^2\Bigr) \\ & \leftstackrel{\eqref{eqn:perpgradphi},\eqref{eqn:fL2norm}}{\le}C \bigl(1+R^2T^2+ \frac{1}{(1-e^{-\sqrt{m}T})^4}+ \frac{R^2}{m(1-e^{-\sqrt{m}T})^4}\bigr)\|f\|_{L^2(\lambda \otimes\mu)}^2. \qedhere
    \end{align*}

\end{proof}

\smallskip

	We are now ready to prove the main results of the paper. The proof is essentially inspired from that of \cite[Proof of Theorem 3]{armstrong2019variational}. In particular, to retrieve the $L^2(\lambda\otimes\mu;H^{-1}_\kappa)$ norm, we need to construct a test function that is in $L^2(\lambda\otimes\mu;H^1_\kappa)$, which is highly related to the test functions constructed in Lemma \ref{lem:truetestfn}. The differences of these two proofs are: (1) we choose the test functions explicitly $\xi_0=1$ and $\xi_i = v_i$, which are orthogonal to each other and have explicit expressions for up to fourth moments (in particular any first and third moments vanish); (2) Instead of using $\|\wb{\nabla} \Pi_v f\|_{H^{-1}(\lambda\otimes\mu)}$ as an intermediate step, we proceed as \eqref{eqn:fkappaest} and control the $L^2(\lambda\otimes\mu;H^1_\kappa)$ norm of another  explicitly constructed function, in order to minimize the usage of Cauchy-Schwarz inequalities and track the dimension dependence of constants carefully. 
	\begin{proof}[Proof of Theorem \ref{thm:poincare}]
		 Without loss of generality, assume $(f)_{\lambda\otimes \stationary}=0$. 
		which indicates $ (\Pi_v f)_{\lambda\otimes \mu} = 0$. Therefore, we can take $\phi_0, \Phi$ as in Lemma \ref{lem:truetestfn} with $\Pi_v f$ in place of $f$, so that $-\partial_t \phi_0 + \nabla_x^* \nabla_x \Phi = \Pi_v f$. The trick in our following step is to introduce $v$ variable in the calculation. Notice by Gaussianity 
		\begin{equation*}
		    \int_{\RR^d} v_i\ \ud\kappa(v)=0,\qquad \int_{\RR^d} v_iv_j\ \ud\kappa(v)=\delta_{i,j},
		\end{equation*} where $\delta_{i,j}$ is the Kronecker symbol which equals to $1$ if $i=j$ and $0$ otherwise. Thus,
		\begin{equation}\label{eqn:fkappaest}\begin{aligned}
			 \|\Pi_v &f\|_{L^2(\lambda\otimes\mu)}^2 =\int_{I\times\mathbb{R}^d} \Pi_v f(-\partial_t \phi_0 +\nabla_x^*\nabla_x \Phi)\ud \lambda(t)\ud\mu(x)\\ & = \int_{I\times\RR^{2d}}\Pi_v f (-\partial_t \phi_0+v\cdot\nabla_x\phi_0+v\cdot\partial_t \nabla_x \Phi-v\cdot\nabla_x^2\Phi\cdot v+\nabla_x\Phi\cdot\nabla_x U) \ud \lambda(t)\ud\stationary(x,v) \\ & = \int_{I\times\RR^{2d}} f (-\partial_t \phi_0+v\cdot\nabla_x\phi_0+v\cdot\partial_t \nabla_x \Phi-v\cdot\nabla_x^2\Phi\cdot v+\nabla_x\Phi\cdot\nabla_x U) \ud \lambda(t)\ud\stationary(x,v)\\ & \qquad + \int_{I\times\RR^{2d}} (\partial_t \phi_0-v\cdot\nabla_x\phi_0-v\cdot\partial_t \nabla_x \Phi+v\cdot\nabla_x^2\Phi\cdot v-\nabla_x\Phi\cdot\nabla_x U)  (f-\Pi_v f) \ud \lambda(t)\ud\stationary(x,v).
		\end{aligned}
		\end{equation}
		For the first integral on the right hand side, we use integration by parts, where it is important that the test functions $(\phi_0,\nabla_x\Phi)$ have Dirichlet boundary conditions in time:
		\begin{align*}
			& \int_{I\times\RR^{2d}}  f (-\partial_t \phi_0+v\cdot\nabla_x\phi_0+v\cdot\partial_t \nabla_x \Phi-v\cdot\nabla_x^2\Phi\cdot v+\nabla_x\Phi\cdot\nabla_x U) \ud \lambda(t)\ud\stationary(x,v)\\  &\qquad  = \int_{I\times\RR^{2d}} \Big(\partial_t f\phi_0-\partial_t f(v\cdot\nabla_x\Phi)-\phi_0(v\cdot\nabla_x f)+f\phi_0(v\cdot\nabla_x U) \\ & \qquad \qquad +(v\cdot \nabla_x f)(v\cdot \nabla_x \Phi)-f(v\cdot\nabla_x\Phi)(v\cdot\nabla_x U)+f\nabla_x\Phi\cdot\nabla_x U\Big)\ud \lambda(t)\ud\stationary(x,v) \\  &\qquad  = \int_{I\times\RR^{2d}} \Big(\partial_t f\phi_0-\partial_t f(v\cdot\nabla_x\Phi)-\phi_0(v\cdot\nabla_x f)+\phi_0(\nabla_v f\cdot\nabla_x U) \\ & \qquad\qquad +(v\cdot \nabla_x f)(v\cdot \nabla_x\Phi)-\nabla_v\cdot((v\cdot \nabla_x\Phi)f \nabla_x U)+ f\nabla_x\Phi\cdot\nabla_x U\Big)\ud \lambda(t)\ud\stationary(x,v)   \\ &\qquad = \int_{I\times\RR^{2d}} \Big((\partial_t f-v\cdot \nabla_x f+\nabla_x U\cdot \nabla_v f)(\phi_0-v\cdot\nabla_x\Phi)\Big)\ud \lambda(t)\ud\stationary(x,v) \\ &\qquad  \le \|\partial_t f-\opLham f\|_{L^2(\lambda\otimes \mu;H^{-1}_\kappa)}\|\phi_0-v\cdot\nabla_x\Phi\|_{L^2(\lambda\otimes \mu;H^1_\kappa)}. \end{align*}
		We further estimate the term $\|\phi_0-v\cdot\nabla_x\Phi\|_{L^2(\lambda\otimes \mu;H^1_\kappa)}$ by explicit integration, noticing $(\phi_0,\Phi)$ do not depend on $v$ so that explicit moments of $v$ can be directly calculated:
		\begin{align*}
		    \|\phi_0-v\cdot\nabla_x\Phi\|_{L^2(\lambda\otimes \mu;H^1_\kappa)}^2 &= \int_{I\times \mathbb{R}^d} \|\phi_0-v\cdot\nabla_x\Phi\|_{H^1_\kappa}^2 \ud \lambda(t)\ud\mu(x) \\ & = \int_{I\times \mathbb{R}^d} \Big(\|\phi_0-v\cdot\nabla_x\Phi\|_{L^2_\kappa}^2+\|\nabla_v(\phi_0-v\cdot\nabla_x\Phi)\|_{L^2_\kappa}^2\Big) \ud \lambda(t)\ud\mu(x) \\ & = \int_{I\times \mathbb{R}^d} \Big(\int_{\mathbb{R}^d}(\phi_0-v\cdot\nabla_x\Phi)^2\ud \kappa(v)+\int_{\mathbb{R}^d}|\nabla_x\Phi|^2\ud \kappa(v)\Big)\ud \lambda(t)\ud\mu(x) \\ & = \int_{I\times \mathbb{R}^d} \big( \phi_0^2+2|\nabla_x\Phi|^2 \big)\ud \lambda(t)\ud\mu(x) \\ & \stackrel{\eqref{eqn:estphi}}{\le} C\big(\frac{1}{m(1-e^{-\sqrt{m} T})^2}+T^2\big)\|\Pi_v f\|_{L^2(\lambda\otimes\mu)}^2.
		\end{align*}
		For the second integral in \eqref{eqn:fkappaest}, we estimate again by explicit expansion in $v$, which is possible since we have explicit up to fourth moments of $v$:
		\begin{align*}
		 \|\partial_t & \phi_0-v\cdot\nabla_x\phi_0-v\cdot\partial_t \nabla_x\Phi+v\cdot\nabla^2_x\Phi \cdot v-\nabla_x\Phi\cdot\nabla_x U\|_{L^2(\lambda \otimes \stationary)}^2 \\ & = \int_{I\times\RR^{2d}} (\partial_t  \phi_0-v\cdot\nabla_x\phi_0-v\cdot\partial_t \nabla_x\Phi+v\cdot\nabla^2_x\Phi \cdot v-\nabla_x\Phi\cdot\nabla_x U)^2\ud \lambda(t)\ud\stationary(x,v)  \\ & = \int_{I\times\RR^{2d}} \Big((\partial_t\phi_0-\nabla_x\Phi\cdot\nabla_x U)^2-2(\partial_t\phi_0-\nabla_x\Phi\cdot\nabla_x U) (v\cdot\nabla_x\phi_0)-2(\partial_t\phi_0-\nabla_x\Phi\cdot\nabla_x U) (v\cdot\partial_t\nabla_x\Phi) \\ & \qquad +(v\cdot\nabla_x \phi_0)^2+(v\cdot\partial_t \nabla_x\Phi)^2+2(\partial_t\phi_0-\nabla_x\Phi\cdot\nabla_x U)v\cdot \nabla_x^2 \Phi\cdot v+2(v\cdot\partial_t \nabla_x\Phi)(v\cdot\nabla_x\phi_0) \\ & \qquad  +(v\cdot \nabla_x^2 \Phi \cdot v)^2-2(v\cdot\partial_{t} \nabla_x\Phi)(v\cdot\nabla_x^2\Phi\cdot v)-2(v\cdot \partial_{x_k} \phi_0)(v\cdot\nabla_x^2 \Phi\cdot v) \Big)\ud \lambda(t)\ud\stationary(x,v)  \\ & = \int_{I\times\RR^{2d}} \Big((\partial_t\phi_0-\nabla_x\Phi\cdot\nabla_x U)^2+\sum_{i} v_i^2\big((\partial_{x_i}\phi_0)^2+(\partial_t \partial_{x_i}\Phi)^2+2\partial_{x_i}\phi_0\partial_t \partial_{x_i}\Phi\big) \\ & \qquad +2(\partial_t\phi_0-\nabla_x\Phi\cdot\nabla_x U)\sum_{i} v_i^2 \partial_{x_i x_i}\Phi  + \sum_{i} v_i^4 (\partial_{x_i x_i}\Phi_i)^2+2\sum_{i\neq j}v_i^2v_j^2(\partial_{x_ix_j} \Phi)^2\\ & \qquad+\sum_{i\neq j} v_i^2v_j^2 \partial_{x_ix_i}\Phi\partial_{x_jx_j}\Phi\Big) \ud \lambda(t)\ud\stationary(x,v)\\ & = \int_{I\times \mathbb{R}^d} \Big((\partial_t\phi_0-\nabla_x\Phi\cdot\nabla_x U)^2+ |\nabla_x\phi_0+\partial_t \nabla_x\Phi|^2+2(\partial_t\phi_0-\nabla_x\Phi\cdot\nabla_x U)\Delta_x\Phi \\ & \qquad  + 3\sum_{i}  (\partial_{x_i x_i}\Phi)^2+2\sum_{i\neq j}(\partial_{x_i x_j} \Phi)^2+\sum_{i\neq j} \partial_{x_ix_i}\Phi\partial_{x_jx_j}\Phi\Big) \ud \lambda(t)\ud\mu(x) \\ & \le  \int_{I\times \mathbb{R}^d} \Big((\partial_t\phi_0-\nabla_x\Phi\cdot\nabla_x U+\Delta_x\Phi)^2+2|\nabla_x \phi_0|^2 + 2 |\wb{\nabla} \nabla_x \Phi|^2\Big) \ud \lambda(t)\ud\mu(x) \\ & \leftstackrel{\eqref{eqn:ufkappa}}{=} \|\Pi_v f\|_{L^2(\lambda\otimes\mu)}^2+2\|  \nabla_x \phi_0\|^2_{L^2(\lambda\otimes\mu)}+2\| \wb{\nabla} \nabla_x \Phi\|^2_{L^2(\lambda\otimes\mu)} \\ &  \leftstackrel{\eqref{eqn:estdiv}}{\le} C\bigl(1+R^2T^2+ \frac{1}{(1-e^{-\sqrt{m}T})^4}+ \frac{R^2}{m(1-e^{-\sqrt{m}T})^4}\bigr)\|\Pi_v f\|_{L^2(\lambda\otimes\mu)}^2.
		\end{align*}
		Combining the above estimates, we arrive at  \begin{align*}
			 \|\Pi_v & f\|_{L^2(\lambda\otimes\mu)}^2 \\ &\le  \|\partial_t f-\opLham f\|_{L^2(\lambda\otimes \mu;H^{-1}_\kappa)}\|\phi_0-v\cdot\nabla_x\Phi\|_{L^2(\lambda\otimes \mu;H^1_\kappa)} \\ & \qquad +  \|\partial_t \phi_0-v\cdot\nabla_x\phi_0-v\cdot\partial_t \nabla_x\Phi+v\cdot\nabla_x^2\Phi\cdot v-\nabla_x\Phi\cdot\nabla_x U\|_{L^2(\lambda \otimes \stationary)}\|f-\Pi_v f\|_{L^2(\lambda \otimes \stationary)} \\ & \le C\Bigl(\big(\frac{1}{\sqrt{m}(1-e^{-\sqrt{m} T})}+T\big) \|\partial_t f-\opLham f\|_{L^2(\lambda\otimes \mu;H^{-1}_\kappa)} \|\Pi_v f\|_{L^2(\lambda\otimes\mu)}\\ & \qquad \qquad +\bigl(1+RT+ \frac{1}{(1-e^{-\sqrt{m}T})^2}+ \frac{R}{\sqrt{m}(1-e^{-\sqrt{m}T})^2}\bigr) \|(\id-\Pi_v)f\|_{L^2(\lambda \otimes \stationary)}  \|\Pi_v f\|_{L^2(\lambda\otimes\mu)}\Bigr).
		\end{align*}
		Finally \begin{align*}
			& \|f\|_{L^2(\lambda \otimes \stationary)} \le \|(\id-\Pi_v )f\|_{L^2(\lambda \otimes \stationary)} +  \|\Pi_v f\|_{L^2(\lambda\otimes\mu)}  \\  & \qquad  \le   C\Bigl(\big(\frac{1}{\sqrt{m}(1-e^{-\sqrt{m} T})}+T\big)\|\partial_t f-\opLham f\|_{L^2(\lambda\otimes \mu;H^{-1}_\kappa)}  \\ & \qquad \qquad + \bigl(1+RT+ \frac{1}{(1-e^{-\sqrt{m}T})^2}+ \frac{R}{\sqrt{m}(1-e^{-\sqrt{m}T})^2}\bigr)\|(\id-\Pi_v) f\|_{L^2(\lambda \otimes \stationary)}\Bigr), 
		\end{align*}
		as claimed. 
	\end{proof}
	
	With Theorem \ref{thm:poincare}, we are now able to prove exponential relaxation to equilibrium claimed in Theorem \ref{thm::decayrate}, which essentially follows from a standard energy estimate.
	
	\begin{proof}[Proof of Theorem \ref{thm::decayrate}]
		We first notice that the solution $f\in H^1_{hyp}((0,T)\otimes\mu)$ for all $T>0$. Indeed, as long as $f_0\in L^2(\mu;H^1_\kappa)$, we have $f(t,\cdot,\cdot) \in L^2(\mu;H^1_\kappa)$ for any $t>0$ (see for example \cite[Theorem 35]{villani_hypocoercivity_2009}), and hence $\partial_t f -\opLham f = -\gamma \nabla_v^*\nabla_v f \in L^2(\lambda\otimes \mu;H_\kappa^{-1})$. We also have that (\ref{eqn:meanzero}) implies $$\int_{\mathbb{R}^d\times \mathbb{R}^d} f(t,x,v) \ud \stationary(x,v)=0$$ for all $t\in(0,T)$. This follows from  $$\dfrac{\ud}{\ud t}\int_{\mathbb{R}^d\times \mathbb{R}^d} f(t,x,v)\ud \stationary(x,v)=0,$$ using the equation \eqref{eqn::opL} and integration by parts.
		\smallskip
		
		For every $0<s<t$, we have the typical energy estimate (hereafter we use $L^2((s,t)\otimes\stationary)$ to denote $L^2(\lambda_{(s,t)}\otimes\stationary)$): \begin{equation}\label{eqn:typicalengest}
			\|f(t,\cdot)\|_{L^2(\stationary)}^2-\|f(s,\cdot)\|_{L^2(\stationary)}^2 =-2\gamma \|\nabla_v f\|_{L^2((s,t)\otimes \stationary)}^2.
		\end{equation}
		In particular, \begin{equation}\label{eqn:nonincL2est}
			\text{ the mapping } t\mapsto \|f(t,\cdot)\|_{L^2(\stationary)}^2 \text{ is nonincreasing.}
		\end{equation}
		Since by equation \eqref{eqn::opL}, \begin{equation*}
			-\gamma \nabla_v^*\nabla_v f=\partial_t f-\opLham f,
		\end{equation*}
		we have $$\|\partial_t f-\opLham f\|_{L^2((s,t)\otimes\mu,H^{-1}_\kappa)}=\gamma \|\nabla_v^*\nabla_v f\|_{L^2((s,t)\otimes\mu,H^{-1}_\kappa)}\le \gamma\|\nabla_v f\|_{L^2((s,t)\otimes\stationary)}. $$
		  Now fix $T$ to be the length of the time interval. Denote $b_1=C(\frac{1}{\sqrt{m}(1-e^{-\sqrt{m} T})}+T)$ and $b_2=C(1+RT+ \frac{1}{(1-e^{-\sqrt{m}T})^2}+ \frac{R}{\sqrt{m}(1-e^{-\sqrt{m}T})^2})$, and thus by Theorem \ref{thm:poincare}, \eqref{eqn:typicalengest} and \eqref{eqn:nonincL2est}, and Gaussian Poincar\'e inequality
		  \begin{equation*}
		      \|(\id-\Pi_v)f\|_{L^2(\lambda \otimes \stationary)} \le \|\nabla_v f\|_{L^2(\lambda \otimes \stationary)},
		  \end{equation*}we have for time stamps $t_k =kT$ \begin{align*}
		&	\|f(t_k,\cdot)\|_{L^2(\stationary)}^2  -\|f(t_{k-1},\cdot)\|_{L^2(\stationary)}^2  \\ & \qquad \le -\dfrac{2\gamma}{(b_1\gamma+b_2)^2}\Bigl(b_2\|\nabla_v f\|_{L^2((t_{k-1},t_k)\otimes\stationary)}+b_1\|\partial_t f-\opLham f\|_{L^2((t_{k-1},t_k)\otimes\mu,H^{-1}_\kappa)}\Bigr)^2  \\ & \qquad \le -\dfrac{2\gamma}{(b_1\gamma+b_2)^2}\Bigl(b_2\|(\id-\Pi_v) f\|_{L^2((t_{k-1},t_k)\otimes\stationary)}+b_1\|\partial_t f-\opLham f\|_{L^2((t_{k-1},t_k)\otimes\mu,H^{-1}_\kappa)}\Bigr)^2 \\ & \qquad \le -\dfrac{2\gamma}{(b_1\gamma+b_2)^2}\|f\|_{L^2((t_{k-1},t_k)\otimes\stationary)}^2  \le -\dfrac{2\gamma T }{(b_1\gamma+b_2)^2}\|f(t_k,\cdot)\|_{L^2(\stationary)}^2.
		\end{align*}
  Now for any $t>0$, we pick
  the integer $k$ satisfying $t_k\le t < t_{k+1}$, so that $\|f(t,\cdot)\|_{L^2(\stationary)} \le \|f(t_k,\cdot)\|_{L^2(\stationary)}$. Applying above
  inequality iteratively and using the monoticity \eqref{eqn:nonincL2est}, we obtain
		\begin{align*}		    	\|f(t,\cdot)\|_{L^2(\stationary)}^2 & \le \Bigl(1+\dfrac{2\gamma T}{(b_1\gamma+b_2)^2}\Bigr)^{-k} \|f_0\|_{L^2(\stationary)}^2 \\ & \le \Bigl(1+\dfrac{2\gamma T}{(b_1\gamma+b_2)^2}\Bigr)^{-\frac{t}{T}+1} \|f_0\|_{L^2(\stationary)}^2 \\ & = \Bigl(1+\dfrac{2\gamma T}{(b_1\gamma+b_2)^2}\Bigr)\exp \Bigl(-\frac{t}{T}\log\bigl(1+\dfrac{2\gamma T}{(b_1\gamma+b_2)^2}\bigr)\Bigr) \|f_0\|_{L^2(\stationary)}^2.
		\end{align*}  
	The prefactor \[1+\dfrac{2\gamma T}{(b_1\gamma
      +b_2)^2} \le C\Bigl(1+ \frac{\gamma T}{\bigl(\frac{\gamma}{\sqrt{m}}+\gamma T+1\bigr)^2}\Bigr)\] is bounded above by a constant.	Using
  $\log(1+x) \ge \frac{1}{C}x$ for $x \in [0,  \frac{1}{C}]$ for some universal constant $C$, and then pick $T=\frac{1}{\sqrt{m}}$, this yields exponential decay with rate
  \begin{equation*}
    \nu\ge C\frac{\gamma}{(b_1\gamma+b_2)^2} \ge C \frac{\gamma m}{(\gamma +R+\sqrt{m})^2},
  \end{equation*} which is precisely \eqref{eqn:exprate}.
	\end{proof}
	
	\section*{Acknowledgment}
	This research is supported in part by National Science Foundation
    via grants DMS-1454939 and CCF-1910571. We would like to thank Rong Ge, Yulong Lu, Jonathan Mattingly, Jean-Christophe Mourrat, and Gabriel Stoltz for helpful discussions, and thank Felix Otto for discussions and providing an idea leading to the proof of Lemma~\ref{lem:truetestfn}. LW would also like to thank Scott
    Armstrong \cite{335599} and Nicola Gigli \cite{335608} for
    answering our question on MathOverflow, which lead to our proof of
    Lemma \ref{lem:ellreg} (ii).

	 \section*{Data Availability Statement}
	 This manuscript has no associated data.
	
		 \section*{Conflict of Interest Statement}
		 The authors have no conflict of interest.

	\appendix
	\section{The decay rate for isotropic quadratic potential}\label{app::isoquad}
	For isotropic quadratic potential, an explicit expression for the spectral gap of $\opL$ is available (thus also the decay rate in \eqref{eqn::exp_decay}). Note that while the result is stated for $d = 1$, it trivially extends to arbitrary dimension for isotropic quadratic potential as different coordinates are independent. The spectrum is also explicitly known for $V=0$ and $x\in \mathbb{T}^d$ on a torus, see \cite{kozlov_effective_1989}.
	\begin{theorem}[{\cite[(10.83)]{risken1989fokker}, \cite[Theorem 3.1]{metafune2002spectrum}}]
		When $U(x) = \frac{m}{2}\abs{x}^2$, $d = 1$, the spectrum of the operator $-\opL$ is given by
		\begin{align*}
			\left\{\lambda_{i,j} :=  \frac{\gamma}{2} (i + j) + \frac{\sqrt{\gamma^2 - 4 m}}{2}(i-j), \qquad i,j = 0, 1, 2, \cdots.\right\}.
		\end{align*}
	\end{theorem}
	
	Let $\lambda_{\text{exact}}$ be the spectral gap for the real component of $\{\lambda_{i,j}\}_{i,j\ge 0}$.
	Notice that the spectral gap is always achieved when $i = 0$ and $j = 1$, thus 
	\begin{align}
		\label{eqn::decayrate_guassian}
		\lambda_{\text{exact}} = \Re\Bigl(\frac{\gamma}{2} - \frac{\sqrt{\gamma^2-4m}}{2}\Bigr). 
	\end{align}
	
	\begin{corollary}
		For any dimension $d$, for isotropic potential $U(x) = \frac{m}{2}\abs{x}^2$, \eqref{eqn::exp_decay} holds with the decay rate $\lambda_{\text{exact}}$.
	\end{corollary}

	\section{The DMS hypocoercive estimation}
	\label{sec::DMS}
	
	In this section, we will revisit the decay rate by DMS estimation \cite{dolbeault_hypocoercivity_2009,dolbeault_hypocoercivity_2015}, adapted and summarized for underdamped Langevin equation in \cite[Sec. 2]{roussel2018spectral}. 
	In the first part of this section, we will review the main result based on \cite{roussel2018spectral}; in addition, we will provide a new estimate of the operator norm of $ \Norm{\opA \opLham (1-\Pi_v)}_{L^2(\stationary) \rightarrow L^2(\stationary)}$, which leads into a more explicit expression of the decay rate. In the second part, we will present the asymptotic analysis of the decay rate with respect to $m$ and $\gamma$, under the assumption that $\nabla_x^2 U \ge -2\, \mathrm{Id}$. 
	
	\subsection{Revisiting the DMS hypocoercive estimation in $L^2(\stationary)$}\label{subsec::comparison_hypocoercivity}
	
	Let us first define an operator
	\begin{align}
		\opA = \left( 1 + (\opLham \Pi_v)^* (\opLham \Pi_v)\right)^{-1} (\opLham \Pi_v)^*
	\end{align}
	and a Lyapunov function $\lyap$ for $\phi(x, v)$ by
	\begin{align}
		\lyap(\phi) = \frac{1}{2} \Norm{\phi}^2_{L^2(\stationary)} -\eps \Inner{\opA \phi}{\phi}_{L^2(\stationary)},
	\end{align}
	where $\eps\in (-1, 1)$ is some quantity depending on $\opL$, to be specified below. The functional $\lyap$ is equivalent to $L^2(\stationary)$ norm in the following sense (see \eg, \cite[Eq. (17)]{roussel2018spectral}),
	\begin{align}
		\frac{1- \abs{\eps}}{2} \Norm{\phi}_{L^2(\stationary)}^2 \le \lyap(\phi) \le \frac{1+\abs{\eps}}{2} \Norm{\phi}_{L^2(\stationary)}^2.
	\end{align}
	\begin{theorem}[{See \cite[Theorem 1]{roussel2018spectral}}] 
		Assume that the \poin{} inequality \eqref{eqn:spatialpoincare} holds and there exists $\Rconst < \infty$ such that 
		\begin{align}\label{eq:opAbound}
			\Norm{\opA \opLham (1-\Pi_v)}_{L^2(\stationary) \rightarrow L^2(\stationary)} \le \Rconst.
		\end{align}
		Suppose $\eps\in (-1, 1)$ is chosen such that $\lambda_{\text{DMS}} = \lambda_{\text{DMS}} (\gamma, m, \Rconst, \eps) > 0$, where
		\begin{align}
			\label{eqn::lambda_DMS}
			\lambda_{\text{DMS}} := \frac{\gamma - \frac{\eps }{1 + m} - \sqrt{\eps^2 (\Rconst + \frac{\gamma}{2})^2 + \left(\gamma - \frac{2 m + 1}{m+1}\eps\right)^2}}{2(1+\abs{\eps})}.
		\end{align}
		Then for any solution $f(t, x, v)$ of \eqref{eqn::opL} with $\int f_0\ \ud \stationary = 0$, we have
		\begin{align*}
			\Norm{f(t, \cdot, \cdot)}_{L^2(\stationary)} \le \sqrt{\frac{1+\abs{\eps}}{1-\abs{\eps}}} \Norm{f_0}_{L^2(\stationary)} e^{- \lambda_{\text{DMS}}\ t}.
		\end{align*}
	\end{theorem}
	Notice that when $\eps = 0$, the rate $\lambda_{\text{DMS}} = 0$, which reduces to the conclusion that $\Norm{f(t, \cdot, \cdot)}_{L^2(\stationary)}$ is non-increasing in time $t$.
	The existence of $\Rconst$ has been studied under fairly general assumptions on the potential $U(x)$ in \cite[Sec. 2]{dolbeault_hypocoercivity_2015}. In the \propref{prop::Rham} below, we provide a simpler estimation of $\Rconst$ only under the assumption of lower bound on Hessian; see the Appendix \ref{sec::proof_lemmas} for its proof.
	The first part of the proof is the same as \cite[Lemma 4]{dolbeault_hypocoercivity_2015}; the simplicity in our approach comes from the application of \emph{Bochner's formula}. It is interesting to observe that $\Rconst$ does not depend on $m$ when $U$ is an isotropic quadratic potential.

	\begin{proposition}
		\label{prop::Rham}
	    Assume there exists $K \in \RR$ such that $\nabla_x^2 U \ge -K\, \mathrm{Id}$ for all $x \in \RR^d$, then we can choose 
		\begin{align}
			\label{eqn::Rconst_bound}
			\Rconst = \sqrt{\max\{K, 2\}}. 
		\end{align}
		such that \eqref{eq:opAbound} is satisfied. 

	    For the isotropic case $U(x) = \frac{m}{2}\abs{x}^2$, we have 
		\begin{align*}
			\Norm{\opA \opLham (1-\Pi_v)}_{L^2(\stationary) \rightarrow L^2(\stationary)} = \sqrt{2}.
		\end{align*} 
		Thus the optimal choice of $\Rconst$ is $\sqrt{2}$ and \eqref{eqn::Rconst_bound} is tight in this case.
	\end{proposition}
	
	As an immediate consequence, if it holds that $\nabla_x^2 U \ge -2\, \mathrm{Id}$, we can take $
			\Rconst = \sqrt{2}$, which is tight for the isotropic case.

	\subsection{Asymptotic analysis of the decay rate}
	In this subsection, we shall assume that $\nabla_x^2 U \ge -2\, \mathrm{Id}$, thus we can choose $\Rconst = \sqrt{2}$, according to the \propref{prop::Rham}.
	To remove the dependence on the parameter $\eps$ and to find the optimal decay rate, let us introduce
	\begin{align}
		\label{eqn::lambda_dms_v2}
		\begin{split}
			\Lambda_{\text{DMS}}(\gamma, m) &:=  \sup_{\eps\in (-1, 1)}\ \lambda_{\text{DMS}}(\gamma, m, \sqrt{2}, \eps)\\
			&= \sup_{\eps\in (-1, 1)}   \frac{\gamma - \frac{\eps }{1 + m} - \sqrt{\eps^2 (\sqrt{2} + \frac{\gamma}{2})^2 + \left(\gamma - \frac{2 m + 1}{m+1}\eps\right)^2}}{2(1+\abs{\eps})},
		\end{split}
	\end{align}
provided that the supremum is not achieved at the boundary \ie, $\eps = 1^{-}$ or $\eps = (-1)^{+}$. 
Observe that 
\begin{itemize}
\item When $\eps = 0$, $\lambda_{\text{DMS}}(\gamma, m, \sqrt{2},  0) = 0$;

\item When $\eps = (-1)^{+}$, $\lambda_{\text{DMS}}(\gamma, m, \sqrt{2}, (-1)^{+}) < 0$.
\end{itemize}
Therefore, the supremum can only be achieved at $\eps = 1^{-}$, or the critical points of the expression on the right hand side of \eqref{eqn::lambda_dms_v2}.  
	In general, it is hard to obtain a simple explicit expression of $\Lambda_{\text{DMS}}(\gamma, m)$. 
	Therefore, we shall consider the following asymptotic regions.

	\begin{proposition} 
		\label{prop::DMS-asymptotic}
		\begin{enumerate}
\item[(\romannumeral1)]  For fixed $m = O(1)$, we have
\begin{equation}
\label{eqn::Lambda_asym_fixed_m}
\Lambda_{\text{DMS}}(\gamma, m) = \left\{\begin{aligned}
	\Bigl(\frac{-(1+m) \sqrt{3m^2+4m+1} + 3m^2+3m+1}{6m^2+8m+3}\Bigr) \gamma +O(\gamma^2), & \qquad \mbox{when} \ \gamma\rightarrow 0;\\
	\frac{4m^2}{(1+m)^2 }\gamma^{-1} + O(\gamma^{-2}), & \qquad \mbox{when} \ \gamma \rightarrow\infty.\\
	\end{aligned}
	\right.		
\end{equation}

\item [(\romannumeral2)] Consider coupled asymptotic regime $\gamma = b\sqrt{m}$ (or equivalently $m =\left(\gamma/b\right)^2$) for some $b = O(1)$, 
			we have
			\begin{equation}
				\label{eqn::Lambda_asym_rate}
				\Lambda_{\text{DMS}}(\gamma, m) = 
				\left\{\begin{aligned} 
& \frac{\gamma^5}{2b^4} + O(\gamma^6), & \qquad \text{ when } \gamma\rightarrow 0;\\			
& \frac{4}{\gamma} + O(\gamma^{-2}), & \qquad \text{ when } \gamma\rightarrow \infty.\\
				\end{aligned}\right.
			\end{equation}
		\end{enumerate}
	\end{proposition}

The proof can be found in Appendix~\ref{sec::proof_lemmas}.
The scaling in the first case is already known in \eg{}, \cite{dolbeault2013exponential, grothaus2016hilbert,roussel2018spectral}; in the above proposition, we simply explicitly calculate the leading order term. The second case is relevant when we choose $\gamma$ to optimize the convergence rate according to $m$ and for the regime $m\to 0$.

	\subsection{Proofs of the Propositions in Appendix}
	\label{sec::proof_lemmas}

	\begin{proof}[Proof of \propref{prop::Rham}]
	
    We first consider the case that Hessian is bounded from below. 
		It is equivalent to consider the operator norm of 
		\begin{align*}
			-(1-\Pi_v) \opLham \opA^* = -(1-\Pi_v) \opLham^2 \Pi_v  \left( 1 + (\opLham \Pi_v)^* (\opLham \Pi_v)\right)^{-1}.
		\end{align*}
		Notice that this operator is supported on $\text{Ran}(\Pi_v)$ from the observation that $\opA = \Pi_v \opA$, it is then equivalent to find the smallest $\Rconst$ such that for any $\phi(x,v)$ with $\Pi_v \phi = \phi$ (\ie, $\phi(x,v) \equiv \phi(x)$ is a function of $x$ only), we have
		\begin{align}
			\label{eqn::Rconst_equiv}
			\Norm{-(1-\Pi_v) \opLham \opA^* \phi}_{L^2(\stationary)}\le \Rconst \Norm{\phi}_{L^2(\stationary)} = \Rconst \Norm{\phi}_{L^2(\mu)}.
		\end{align}
		
Given such a function $\phi$ with $\Pi_v \phi = \phi$, define 
		\[\varphi := \left( 1 + (\opLham \Pi_v)^* (\opLham \Pi_v)\right)^{-1} \phi.\]
	It is easy to check that $\Pi_v \varphi = \varphi$.
	By simplifying the above equation with \eqref{eqn::Lham_Lfd} and \eqref{eqn::opLham_opLfd_v2}, 
		\begin{align}
			\label{eqn::wtphi_wtvarphi}
			\phi(x) = \varphi(x) -  \Delta_x \varphi(x) + \nabla_x U(x) \cdot \nabla_x  \varphi = \varphi(x) +  \nabla_x^* \nabla_x \varphi(x).
		\end{align}
		
		Furthermore, by some straightforward calculation, we have
		\begin{align*}
			-(1-\Pi_v) \opLham \opA^* \phi = - (1-\Pi_v) \opLham^2 \Pi_v \varphi 
			= - \sum_{i,j} (v_i v_j -  \delta_{i,j}) \partial_{x_i, x_j} \varphi.
		\end{align*}
		Thus 
		\begin{align*}
			\Norm{-(1-\Pi_v) \opLham \opA^* \phi}_{L^2(\stationary)}^2 &= 
			\int \Bigl(\sum_{i,j} (v_i v_j -  \delta_{i,j}) \partial_{x_i, x_j} \varphi\Bigr)^2\ud \stationary\\
		&= 2  \sum_{i,j} \int \left(\partial_{x_i, x_j} {\varphi}\right)^2\ud\mu.
		\end{align*}
		Then by Bochner's formula, 
		\begin{align*}
			 \lVert -(1-\Pi_v) \opLham & \opA^* (\phi) \rVert_{L^2(\stationary)}^2 \\
			=&\ 2 \int \nabla_x \varphi \cdot \nabla_x \nabla_x^* \nabla_x {\varphi} - \nabla_x {\varphi} \cdot \nabla_x^2 U \nabla_x {\varphi} - \nabla_x^* \nabla_x \left(\frac{\abs{\nabla_x {\varphi}}^2}{2}\right)\ud\mu \\
		=&\ 2 \int \abs{\nabla_x^*\nabla_x {\varphi}}^2 - \nabla_x {\varphi} \cdot \nabla_x^2 U \nabla_x {\varphi}\ud\mu \\
			\le &\ 2 \left(\int \abs{\nabla_x^*\nabla_x {\varphi}}^2\ud\mu + K \int \abs{\nabla_x {\varphi}}^2\ud\mu \right)\\
			\le &\ \max\left\{ K, 2\right\} \left(\int \abs{\nabla_x^* \nabla_x {\varphi}}^2\ud\mu + 2 \int \abs{\nabla_x {\varphi}}^2\ud\mu\right). 
		\end{align*}
		From \eqref{eqn::wtphi_wtvarphi}, we have 
		\begin{align*}
			\Norm{{\phi}}_{L^2(\mu)}^2 &= \int {\varphi}^2 + 2  {\varphi}\ \nabla_x^* \nabla_x {\varphi} +  \Abs{\nabla_x^*\nabla_x {\varphi}}^2\  \rd\mu \\ 
			&\ge  2 \int \abs{\nabla_x{\varphi}}^2\ud\mu
 + \int \abs{\nabla_x^* \nabla_x {\varphi}}^2\ud\mu. 		\end{align*}
		By combining the last two equations, 
		\begin{align*}
			\Norm{-(1-\Pi_v) \opLham \opA^* (\phi)}_{L^2(\stationary)}^2 &\le \max\{K, 2\} \Norm{{\phi}}_{L^2(\mu)}^2,
		\end{align*}
		which yields \eqref{eqn::Rconst_bound}. 

    \medskip 
    
    We now consider the isotropic case. 
Recall that the operator norm of $\opA \opLham (1-\Pi_v)$ is the smallest $\Rconst$ such that  \eqref{eqn::Rconst_equiv} holds.  
		Let us consider the elliptic PDE \eqref{eqn::wtphi_wtvarphi}. By the choice $U(x) = \frac{m}{2}\abs{x}^2$, 
		\begin{align*}
			{\phi}(x) = \Bigl(1 + m (x - \frac{1}{m} \nabla_x) \cdot \nabla_x \Bigr) {\varphi}(x).
		\end{align*}
		Then by rescaling the variable $x = \frac{y}{\sqrt{m}}$ and rescaling the functions $\wb{\phi}(y) := {\phi}(x) = \phi\bigl(\frac{y}{\sqrt{m}}\bigr)$, $\wb{\varphi}(y) := {\varphi}(x) = {\varphi}\bigl(\frac{y}{\sqrt{m}}\bigr)$, we have
		\begin{align}
			\wb{\phi}(y) = \Bigl( 1 + m (y - \nabla_y) \cdot \nabla_y\Bigr) \wb{\varphi}(y). 
		\end{align}
		In addition, by rewriting \eqref{eqn::Rconst_equiv}, we need to find the smallest $\Rconst$ such that
		\begin{align}
			\label{eqn::Rconst_v3}
			2 m^2 \sum_{i,j} \int \Abs{\partial_{y_i, y_j} \wb{\varphi}(y)}^2\
			e^{-\frac{\abs{y}^2}{2}}\ud y\le \Rconst^2 \int \wb{\phi}(y)^2 e^{-\frac{\abs{y}^2}{2}}\ud y.
		\end{align}

	Next, let us expand the last equation by probabilists' Hermite polynomials $H_k(z) := (z - \frac{d}{dz})^k \cdot 1$ for integers $k\ge 0$. Recall two important properties 
	\begin{align*}
	    H'_k(z) = k H_{k-1}(z), \qquad \frac{1}{\sqrt{2\pi}}\int H_j(z) H_k(z) e^{-\frac{z^2}{2}}\ud z = k!\, \delta_{j,k}.
	\end{align*} 
Given $\vectbf{n} = (n_1, n_2, \cdots, n_d)$, define
		\begin{align*}
		H_{\vectbf{n}}(y) := H_{n_1}(y_1) H_{n_2}(y_2) \cdots H_{n_d}(y_d).
		\end{align*}
		 By the above properties, it is easy to show that if $\wb{\varphi} = H_{\vectbf{n}}$, then $\wb{\phi} = N_{\vectbf{n}} H_{\vectbf{n}}$, where $N_{\vectbf{n}} := 1 + m \sum_{i} n_i$. Thus if $\wb{\varphi}(y) = \sum_{\vectbf{n}} a_{\vectbf{n}}H_{\vectbf{n}}$, then we have $\wb{\phi} = \sum_{n} a_{\vectbf{n}} N_{\vectbf{n}} H_{\vectbf{n}}$. By such an expansion, \eqref{eqn::Rconst_v3} can be rewritten as 
		\begin{align*}
			2m^2 \sum_{i, j} \sum_{\vectbf{n}} a_{\vectbf{n}}^2  (n_i n_j - \delta_{i,j} n_i) \prod_{k=1}^{d} n_k! \le \Rconst^2 \sum_{\vectbf{n}} a_{\vectbf{n}}^2 N_{\vectbf{n}}^2 \prod_{k=1}^{d} n_k!
		\end{align*}
		Then finding the operator norm of $\opA \opLham (1-\Pi_v)$ is equivalent to finding the smallest $\Rconst$ such that for any $\vectbf{n}$, one has
		\begin{align*}
			\sum_{i,j} (n_i n_j - \delta_{i,j} n_i) \le \frac{\Rconst^2}{2 m^2 } N_{\vectbf{n}}^2 \equiv \frac{\Rconst^2}{2 m^2 } (1 + m \sum_i n_i)^2. 
		\end{align*}
		When $n_1 \rightarrow\infty$ and $n_2, n_3, \cdots, n_d = 0$, we know that $\frac{\Rconst^2}{2} \ge 1$. Also observe that 
			\begin{align*}
				\sum_{i,j} (n_i n_j - \delta_{i,j} n_i) \le \Bigl(\sum_{i} n_i\Bigr)^2 = \frac{1}{m^2} \Bigl( m \sum_i n_i\Bigr)^2\le \frac{1}{m^2} \Bigl(1 + m \sum_i n_i\Bigr)^2. 
			\end{align*}
			Therefore, $\frac{\Rconst^2}{2} = 1$ is sufficient. 
		
		In summary, $\Norm{\opA \opLham (1-\Pi_v)}_{L^2(\stationary) \rightarrow L^2(\stationary)}  = \sqrt{2}$ and the optimal choice of $\Rconst$ is $\sqrt{2}$.
	\end{proof}

	\begin{proof}[Proof of \propref{prop::DMS-asymptotic}]
    We used \textsf{Maple} software to help verify the asymptotic expansion. 
		
{\noindent {\bf Part (\romannumeral1): $m = O(1)$.}}
		\begin{itemize}[wide]
			\item 
			{\noindent {\bf (when $\gamma\rightarrow 0$)}}. Via asymptotic expansion, we have 
			\begin{align*}
			\lambda_{\text{DMS}}(\gamma, m, \sqrt{2}, 1^{-}) = - \frac{1+\sqrt{6 m^2 + 8m + 3}}{4(1+m)} + O(\gamma) < 0.
			\end{align*}
Thus the supremum is not obtained at $\eps = 1^{-}$. Then let us consider critical points within the domain $(-1, 1)$, whose asymptotic expansions are
		\begin{align*}
		\eps_{\pm} =  \frac{(6m^2+5m+1 \pm \sqrt{3m^2+4m+1})(1+m)}{18m^3+30m^2+17m+3} \gamma + O(\gamma^2) > 0.
		\end{align*}
		After comparison, the larger decay rate is obtained 
		at $\eps_{-}$ with the value in \eqref{eqn::Lambda_asym_fixed_m}. 
		
		\item {\noindent {\bf (when $\gamma\rightarrow \infty$)}.}
		Similarly, via asymptotic expansion, we have 
		\begin{align*}
		\lambda_{\text{DMS}}(\gamma, m, \sqrt{2}, 1^{-}) = -\frac{\frac{\sqrt{5}}{2}-1}{4}\gamma + O(1) < 0.
		\end{align*}
		Thus we need to consider the critical points. It turns out, there is only one critical point within the domain $(-1, 1)$, which is $\epsilon = \frac{8m}{1+m}\gamma^{-1} + O(\gamma^{-2})$ with the decay rate in \eqref{eqn::Lambda_asym_fixed_m}. 
	\end{itemize}

{\noindent {\bf Part (\romannumeral2): $\gamma = b\sqrt{m}$ with $b=O(1)$.}}
\begin{itemize}[wide]
\item	{\noindent {\bf (when $\gamma\rightarrow 0$)}.} 
Via asymptotic expansion, one could check that 
\begin{align*}
\lambda_{\text{DMS}}(\gamma, m=(\gamma/b)^2, \sqrt{2}, 1^{-}) = -\frac{1+\sqrt{3}}{4} + O(\gamma) < 0.
\end{align*}
Thus, we only need to consider the decay rate at critical points, which are given by
\begin{align*}
\eps_1 = \frac{\gamma^3}{b^2} + O(\gamma^4), \qquad
\eps_2 = \frac{2}{3} \gamma + O(\gamma^2). 
\end{align*}
and the associated decay rates are
\begin{align*}
\lambda_{\text{DMS}}(\gamma, m=(\gamma/b)^2, \sqrt{2}, \eps_1)   &= \frac{\gamma^5}{2b^4}  + O(\gamma^{6})> 0;\\
\lambda_{\text{DMS}}(\gamma, m=(\gamma/b)^2, \sqrt{2}, \eps_2)  &= -\frac{1}{3} \gamma + O(\gamma^2) < 0.
\end{align*}
Therefore, the optimal decay rate is obtained at $\eps_1$, which gives \eqref{eqn::Lambda_asym_rate}. 
			
\item 
{\noindent {\bf (when $\gamma\rightarrow\infty$)}.}
Via asymptotic expansion, one could obtain
\begin{align*}
\lambda_{\text{DMS}}(\gamma, m=(\gamma/b)^2, \sqrt{2}, 1^{-}) =  -\frac{\sqrt{5}-2}{8}\gamma + O(1) < 0.
\end{align*}
Thus the supremum in \eqref{eqn::lambda_dms_v2} cannot be obtained at $\eps = 1^{-}$. 
Then, let us look at the critical points. It turns out there is only one within the interval $(-1, 1)$, which is 
$\eps_1 = \frac{8}{\gamma} + O(\gamma^{-2})$. 
The optimal decay rate must be achieved at $\eps_1$, with the expression given in \eqref{eqn::Lambda_asym_rate}.		\qedhere
\end{itemize}
\end{proof}
	
\bibliographystyle{amsplain}
\bibliography{samplingpoincare}

\providecommand{\bysame}{\leavevmode\hbox to3em{\hrulefill}\thinspace}
\providecommand{\MR}{\relax\ifhmode\unskip\space\fi MR }
\providecommand{\MRhref}[2]{%
  \href{http://www.ams.org/mathscinet-getitem?mr=#1}{#2}
}
\providecommand{\href}[2]{#2}
\begin{thebibliography}{10}

\bibitem{armstrong2019variational}
Dallas Albritton, Scott Armstrong, Jean-Christophe Mourrat, and Matthew Novack,
  \emph{Variational methods for the kinetic {F}okker-{P}lanck equation}, arXiv
  preprint arXiv:1902.04037 (2019).

\bibitem{andrieu2021hypocoercivity}
Christophe Andrieu, Alain Durmus, Nikolas N{\"u}sken, and Julien Roussel,
  \emph{Hypocoercivity of piecewise deterministic {M}arkov process-{Monte
  Carlo}}, The Annals of Applied Probability \textbf{31} (2021), no.~5,
  2478--2517.

\bibitem{335599}
Scott Armstrong, \emph{Answer to ``{E}lliptic regularity with {G}ibbs measure
  satisfying {B}akry-{E}mery condition''}, MathOverflow,
  https://mathoverflow.net/q/335599 (version: 2019-07-06).

\bibitem{bakry2008rate}
Dominique Bakry, Patrick Cattiaux, and Arnaud Guillin, \emph{Rate of
  convergence for ergodic continuous {M}arkov processes: {L}yapunov versus
  {P}oincar{\'e}}, Journal of Functional Analysis \textbf{254} (2008), no.~3,
  727--759.

\bibitem{bakry1985diffusions}
Dominique Bakry and Michel {\'E}mery, \emph{Diffusions hypercontractives},
  S{\'e}minaire de Probabilit{\'e}s XIX 1983/84, Springer, 1985, pp.~177--206.

\bibitem{bakry_analysis_2014}
Dominique Bakry, Ivan Gentil, and Michel Ledoux, \emph{Analysis and geometry of
  {Markov} diffusion operators}, Springer, Cham; New York, 2014.

\bibitem{baudoin_wasserstein_2016}
Fabrice Baudoin, \emph{Wasserstein contraction properties for hypoelliptic
  diffusions}, arXiv:1602.04177 [math] (2016).

\bibitem{baudoin_bakry-emery_2013}
\bysame, \emph{Bakry--{\'e}mery meet {V}illani}, Journal of functional analysis
  \textbf{273} (2017), no.~7, 2275--2291.

\bibitem{baudoin2021gamma}
Fabrice Baudoin, Maria Gordina, and David~P Herzog, \emph{Gamma calculus beyond
  {V}illani and explicit convergence estimates for {L}angevin dynamics with
  singular potentials}, Archive for Rational Mechanics and Analysis
  \textbf{241} (2021), no.~2, 765--804.

\bibitem{bernard2022hypocoercivity}
{\'E}tienne Bernard, Max Fathi, Antoine Levitt, and Gabriel Stoltz,
  \emph{Hypocoercivity with {S}chur complements}, Annales Henri Lebesgue
  \textbf{5} (2022), 523--557.

\bibitem{bogovskii1979solution}
Mikhail~Evgen'evich Bogovskii, \emph{Solution of the first boundary value
  problem for the equation of continuity of an incompressible medium}, Doklady
  Akademii Nauk, vol. 248, Russian Academy of Sciences, 1979, pp.~1037--1040.

\bibitem{camrud2021weighted}
Evan Camrud, David~P Herzog, Gabriel Stoltz, and Maria Gordina, \emph{Weighted
  {$L^2$}-contractivity of {L}angevin dynamics with singular potentials},
  Nonlinearity \textbf{35} (2021), no.~2, 998.

\bibitem{cattiaux2019entropic}
Patrick Cattiaux, Arnaud Guillin, Pierre Monmarch{\'e}, and Chaoen Zhang,
  \emph{Entropic multipliers method for {L}angevin diffusion and weighted log
  {S}obolev inequalities}, Journal of Functional Analysis \textbf{277} (2019),
  no.~11, 108288.

\bibitem{conrad2010construction}
Florian Conrad and Martin Grothaus, \emph{Construction, ergodicity and rate of
  convergence of ${N}$-particle {L}angevin dynamics with singular potentials},
  Journal of Evolution Equations \textbf{10} (2010), no.~3, 623--662.

\bibitem{cooke2011geometric}
Ben Cooke, David~P Herzog, Jonathan~C Mattingly, Scott~A McKinley, and Scott~C
  Schmidler, \emph{Geometric ergodicity of two--dimensional {H}amiltonian
  systems with a {Lennard--Jones}--like repulsive potential}, Communications in
  Mathematical Sciences \textbf{15} (2017), no.~7, 1987--2025.

\bibitem{dalalyan2018sampling}
Arnak~S Dalalyan and Lionel Riou-Durand, \emph{On sampling from a log-concave
  density using kinetic {L}angevin diffusions}, Bernoulli \textbf{26} (2020),
  no.~3, 1956--1988.

\bibitem{dolbeault2013exponential}
Jean Dolbeault, Axel Klar, Cl{\'e}ment Mouhot, and Christian Schmeiser,
  \emph{Exponential rate of convergence to equilibrium for a model describing
  fiber lay-down processes}, Applied Mathematics Research eXpress \textbf{2013}
  (2013), no.~2, 165--175.

\bibitem{dolbeault_hypocoercivity_2009}
Jean Dolbeault, Cl{\'e}ment Mouhot, and Christian Schmeiser,
  \emph{Hypocoercivity for kinetic equations with linear relaxation terms},
  Comptes Rendus Mathematique \textbf{347} (2009), no.~9, 511--516.

\bibitem{dolbeault_hypocoercivity_2015}
\bysame, \emph{Hypocoercivity for linear kinetic equations conserving mass},
  Transactions of the American Mathematical Society \textbf{367} (2015), no.~6,
  3807--3828.

\bibitem{eberle_couplings_2019}
Andreas Eberle, Arnaud Guillin, and Raphael Zimmer, \emph{Couplings and
  quantitative contraction rates for {Langevin} dynamics}, The Annals of
  Probability \textbf{47} (2019), no.~4, 1982--2010.

\bibitem{eckmann_spectral_2003}
J.-P. Eckmann and M.~Hairer, \emph{Spectral properties of hypoelliptic
  operators}, Communications in Mathematical Physics \textbf{235} (2003),
  no.~2, 233--253.

\bibitem{evans2010partial}
Lawrence~C Evans, \emph{Partial differential equations}, vol.~19, American
  Mathematical Soc., 2010.

\bibitem{galdi2011introduction}
Giovanni Galdi, \emph{An introduction to the mathematical theory of the
  {N}avier-{S}tokes equations: Steady-state problems}, Springer Science \&
  Business Media, 2011.

\bibitem{335608}
Nicola Gigli, \emph{Answer to ``{E}lliptic regularity with {G}ibbs measure
  satisfying {B}akry-{E}mery condition''}, MathOverflow,
  https://mathoverflow.net/q/335608 (version: 2019-07-06).

\bibitem{grothaus_hypocoercivity_2014}
Martin Grothaus and Patrik Stilgenbauer, \emph{Hypocoercivity for {Kolmogorov}
  backward evolution equations and applications}, Journal of Functional
  Analysis \textbf{267} (2014), no.~10, 3515--3556.

\bibitem{grothaus2016hilbert}
\bysame, \emph{Hilbert space hypocoercivity for the {L}angevin dynamics
  revisited}, Methods of Functional Analysis and Topology \textbf{22} (2016),
  no.~02, 152--168.

\bibitem{nier2005hypoelliptic}
Bernard Helffer and Francis Nier, \emph{Hypoelliptic estimates and spectral
  theory for {F}okker-{P}lanck operators and {W}itten {L}aplacians}, vol. 1862,
  Springer Science \& Business Media, 2005.

\bibitem{herau_hypocoercivity_2006}
Fr{\'e}d{\'e}ric H{\'e}rau, \emph{Hypocoercivity and exponential time decay for
  the linear inhomogeneous relaxation {Boltzmann} equation}, Asymptotic
  Analysis \textbf{46} (2006), no.~3-4, 349--359. \MR{2215889}

\bibitem{herau_isotropic_2004}
Fr{\'e}d{\'e}ric H{\'e}rau and Francis Nier, \emph{Isotropic hypoellipticity
  and trend to equilibrium for the {Fokker}-{Planck} equation with a
  high-degree potential}, Archive for Rational Mechanics and Analysis
  \textbf{171} (2004), no.~2, 151--218.

\bibitem{herzog2019ergodicity}
David~P Herzog and Jonathan~C Mattingly, \emph{Ergodicity and {L}yapunov
  functions for {L}angevin dynamics with singular potentials}, Communications
  on Pure and Applied Mathematics \textbf{72} (2019), no.~10, 2231--2255.

\bibitem{hooton1981compact}
James~G Hooton, \emph{Compact {S}obolev imbeddings on finite measure spaces},
  Journal of Mathematical Analysis and Applications \textbf{83} (1981), no.~2,
  570--581.

\bibitem{hormander1967hypoelliptic}
Lars H{\"o}rmander, \emph{Hypoelliptic second order differential equations},
  Acta Mathematica \textbf{119} (1967), 147--171.

\bibitem{iacobucci_convergence_2019}
A.~Iacobucci, S.~Olla, and G.~Stoltz, \emph{Convergence rates for
  nonequilibrium {Langevin} dynamics}, Annales math{\'e}matiques du Qu{\'e}bec
  \textbf{43} (2019), no.~1, 73--98.

\bibitem{kolmogoroff1934zufallige}
Andrey Kolmogorov, \emph{Zufallige bewegungen (zur theorie der {B}rownschen
  bewegung)}, Annals of Mathematics (1934), 116--117.

\bibitem{kozlov_effective_1989}
S.~M. Kozlov, \emph{Effective diffusion in the {Fokker}-{Planck} equation},
  Mathematical notes of the Academy of Sciences of the USSR \textbf{45} (1989),
  no.~5, 360--368.

\bibitem{ledoux1994simple}
Michel Ledoux, \emph{A simple analytic proof of an inequality by {P}. {B}user},
  Proceedings of the American Mathematical Society \textbf{121} (1994), no.~3,
  951--959.

\bibitem{ledoux2004spectral}
\bysame, \emph{Spectral gap, logarithmic {S}obolev constant, and geometric
  bounds}, Surveys in Differential Geometry \textbf{9} (2004), no.~1, 219--240.

\bibitem{leimkuhler2020hypocoercivity}
Benedict Leimkuhler, Matthias Sachs, and Gabriel Stoltz, \emph{Hypocoercivity
  properties of adaptive {L}angevin dynamics}, SIAM Journal on Applied
  Mathematics \textbf{80} (2020), no.~3, 1197--1222.

\bibitem{lu2019geometric}
Yulong Lu and Jonathan~C Mattingly, \emph{Geometric ergodicity of {L}angevin
  dynamics with {C}oulomb interactions}, Nonlinearity \textbf{33} (2019),
  no.~2, 675.

\bibitem{ma2021there}
Yi-An Ma, Niladri~S Chatterji, Xiang Cheng, Nicolas Flammarion, Peter~L
  Bartlett, and Michael~I Jordan, \emph{Is there an analog of {N}esterov
  acceleration for gradient-based {MCMC}?}, Bernoulli \textbf{27} (2021),
  no.~3, 1942--1992.

\bibitem{mattingly_ergodicity_2002}
J.~C. Mattingly, A.~M. Stuart, and D.~J. Higham, \emph{Ergodicity for {SDEs}
  and approximations: locally {Lipschitz} vector fields and degenerate noise},
  Stochastic Processes and their Applications \textbf{101} (2002), no.~2,
  185--232.

\bibitem{metafune2002spectrum}
Giorgio Metafune, Diego Pallara, and Enrico Priola, \emph{Spectrum of
  ornstein-uhlenbeck operators in lp spaces with respect to invariant
  measures}, Journal of Functional Analysis \textbf{196} (2002), no.~1, 40--60.

\bibitem{mouhot2006quantitative}
Cl{\'e}ment Mouhot and Lukas Neumann, \emph{Quantitative perturbative study of
  convergence to equilibrium for collisional kinetic models in the torus},
  Nonlinearity \textbf{19} (2006), no.~4, 969.

\bibitem{otto2000generalization}
Felix Otto and C{\'e}dric Villani, \emph{Generalization of an inequality by
  {T}alagrand and links with the logarithmic {S}obolev inequality}, Journal of
  Functional Analysis \textbf{173} (2000), no.~2, 361--400.

\bibitem{pavliotis2014stochastic}
Grigorios~A Pavliotis, \emph{Stochastic processes and applications: diffusion
  processes, the {F}okker-{P}lanck and {L}angevin equations}, vol.~60,
  Springer, 2014.

\bibitem{risken1989fokker}
H~Risken, \emph{Fokker-planck equation: methods of solution and applications},
  Springer series in synergetics (1989).

\bibitem{roussel2018spectral}
Julien Roussel and Gabriel Stoltz, \emph{Spectral methods for {L}angevin
  dynamics and associated error estimates}, ESAIM: Mathematical Modelling and
  Numerical Analysis \textbf{52} (2018), no.~3, 1051--1083.

\bibitem{sason_f_2016}
I.~Sason and S.~Verd{\'u}, \emph{$f$-divergence inequalities}, IEEE
  Transactions on Information Theory \textbf{62} (2016), no.~11, 5973--6006.

\bibitem{stoltz_langevin_2018}
Gabriel Stoltz and Zofia Trstanova, \emph{Langevin dynamics with general
  kinetic energies}, Multiscale Modeling \& Simulation \textbf{16} (2018),
  no.~2, 777--806.

\bibitem{stoltz_longtime_2018}
Gabriel Stoltz and Eric Vanden-Eijnden, \emph{Longtime convergence of the
  temperature-accelerated molecular dynamics method}, Nonlinearity \textbf{31}
  (2018), no.~8, 3748--3769.

\bibitem{talay2002stochastic}
Denis Talay, \emph{Stochastic {H}amiltonian systems: exponential convergence to
  the invariant measure, and discretization by the implicit {E}uler scheme},
  Markov Process. Related Fields \textbf{8} (2002), no.~2, 163--198.

\bibitem{tropper1977ergodic}
MM~Tropper, \emph{Ergodic and quasideterministic properties of
  finite-dimensional stochastic systems}, Journal of Statistical Physics
  \textbf{17} (1977), no.~6, 491--509.

\bibitem{villani_hypocoercive_2007}
C{\'e}dric Villani, \emph{Hypocoercive diffusion operators}, Bollettino
  dell'Unione Matematica Italiana \textbf{10-B} (2007), no.~2, 257--275.

\bibitem{villani_hypocoercivity_2009}
\bysame, \emph{Hypocoercivity}, Memoirs of the American Mathematical Society
  \textbf{202} (2009), no.~950.

\bibitem{wu2001large}
Liming Wu, \emph{Large and moderate deviations and exponential convergence for
  stochastic damping {H}amiltonian systems}, Stochastic processes and their
  applications \textbf{91} (2001), no.~2, 205--238.

\end{thebibliography}

\end{document}